\documentclass[12pt,reqno]{amsart}
\usepackage[margin=1in]{geometry}
\usepackage{amsmath,amssymb,amsthm,graphicx,amsxtra, setspace}
\usepackage[utf8]{inputenc}
\usepackage{mathrsfs}
\usepackage{hyperref}
\usepackage{upgreek}
\usepackage{mathtools}
\usepackage{xcolor}
\allowdisplaybreaks

\usepackage[pagewise]{lineno}

\usepackage{graphicx,eurosym}
\usepackage{hyperref}
\usepackage{mathtools}

\usepackage[cyr]{aeguill}

\colorlet{darkblue}{blue!50!black}

\hypersetup{
	colorlinks,%
	citecolor=blue,%
	filecolor=red,%
	linkcolor=darkblue,%
	urlcolor=blue,%
	pdfnewwindow=true,%
	pdfstartview={FitH}
}

\colorlet{darkblue}{red!100!black}

\newtheorem{theorem}{Theorem}[section]
\newtheorem{lemma}[theorem]{Lemma}
\newtheorem{proposition}[theorem]{Proposition}

\newtheorem{corollary}[theorem]{Corollary}
\newtheorem{definition}[theorem]{Definition}

\newtheorem{remark}[theorem]{Remark}
\newtheorem{hypothesis}[theorem]{Hypothesis}

\let\originalleft\left
\let\originalright\right
\renewcommand{\left}{\mathopen{}\mathclose\bgroup\originalleft}
\renewcommand{\right}{\aftergroup\egroup\originalright}

\theoremstyle{definition}

\newtheorem{condition}{Condition}[section]

\def\1{\mathbb{T}^2}

\def\T{\mathbb{T}}

\def\wi{\widehat}
\def\vi{\widetilde}

\def\d{\mathrm{d}}

\def\I{\mathrm{I}}
\def\B{\mathrm{B}}
\def\D{\mathrm{D}}
\def\A{\mathrm{A}}

\def\W{\mathrm{W}}
\def\R{\mathbb{R}}
\def\E{\mathbb{E}}
\def\Q{\mathrm{Q}}
\def\H{\mathbb{H}}
\def\V{\mathbb{V}}
\def\e{\epsilon}
\def\2{\mathcal{E}}
\def\L{\mathrm{L}}
\def\u{\boldsymbol{u}}
\def\v{\boldsymbol{v}}

\def\C{\mathrm{C}}

\def\F{\mathrm{F}}
\def\P{\mathbb{P}}
\def\N{\mathbb{N}}
\def\y{y}
\def\G{\mathrm{G}}

\def\Z{\mathbb{Z}}
\def\PP{\mathscr{P}}
\def\U{\mathbb{U}}

\def\x{x}

\def\z{z}
\def\b{\boldsymbol{b}}

\def\S{\mathrm{S}}
\def\M{\mathrm{M}}
\def\EE{\mathcal{E}}
\def\J{\mathrm{J}}
\def\LL{\mathbb{L}}
\def\k{\boldsymbol{k}}
\def\q{\boldsymbol{q}}
\def\b{\mathrm{b}}
\newcommand{\Addresses}{{
		\footnote{
			\noindent \textsuperscript{1,2}Department of Mathematics, Indian Institute of Technology Roorkee-IIT Roorkee,
			Haridwar Highway, Roorkee, Uttarakhand 247667, INDIA.\par\nopagebreak
			\noindent  \textit{e-mail:} \texttt{Manil T. Mohan: maniltmohan@ma.iitr.ac.in, maniltmohan@gmail.com.}
			
			\textit{e-mail:} \texttt{Ankit Kumar: akumar14@mt.iitr.ac.in.}
			
			\noindent \textsuperscript{*}Corresponding author.
			
			\textit{Keywords:} Stochastic Navier-Stokes equation, vorticity, uniform large deviation principle, Wiener process.
			
			Mathematics Subject Classification (2020): Primary 60H15, 60F10; Secondary  35Q30, 37L55, 60H20.

}}}
\begin{document}	
	
	\title[ULDP for 2D-SNSE]{Uniform large deviation principle for the solutions of two-dimensional stochastic Navier-Stokes equations in vorticity form
		\Addresses}

	\author[A. Kumar and M. T. Mohan]
	{Ankit Kumar\textsuperscript{1} and Manil T. Mohan\textsuperscript{2*}}

	\maketitle

\begin{abstract}
	The main objective of this paper is to demonstrate the uniform large deviation principle (UDLP) for the solutions of two-dimensional stochastic Navier-Stokes equations (SNSE) in the vorticity form when perturbed by two distinct types of noises. We first consider an infinite-dimensional additive noise that is white in time and colored in space and then consider a finite-dimensional Wiener process with linear growth coefficient. In order to obtain the ULDP for 2D SNSE in the vorticity form, where the noise is white in time and colored in space, we utilize the existence and uniqueness  result from \emph{B. Ferrario et. al., Stochastic Process. Appl., {\bf 129} (2019), 1568--1604,} and the \textsl{uniform contraction principle}. For the  finite-dimensional multiplicative Wiener noise, we first prove the existence of a unique local mild solution to the vorticity equation using a truncation and fixed point arguments. We then establish  the global existence of the truncated system by deriving a uniform energy estimate for the local mild solution.  By applying stopping time arguments and a version of Skorokhod's representation theorem, we conclude the global existence and uniqueness of a solution to our model. We employ the weak convergence approach to establish the ULDP for the law of the solutions in two distinct topologies. We prove ULDP in the $\mathrm{C}([0,T];\mathrm{L}^p(\mathbb{T}^2))$ topology, for $p>2$, taking into account the uniformity of the initial conditions contained in bounded subsets of $\mathrm{L}^p(\mathbb{T}^2)$. Finally, in $\mathrm{C}([0,T]\times\mathbb{T}^2)$ topology, the uniformity of initial conditions lying in bounded subsets of $\mathrm{C}(\mathbb{T}^2)$ is considered.
\end{abstract}

	\section{Introduction}\label{Sec1}\setcounter{equation}{0}
We consider the following two-dimensional stochastic  Navier-Stokes equations (SNSE):
\begin{equation}\label{1.1}
	\left\{
	\begin{aligned}
		\frac{\partial \u^\e}{\partial t}(t,\x)+(\u^\e(t,\x)\cdot\nabla)\u^\e(t,\x)-&\nu\Delta\u^\e(t,\x)+\nabla p^\e(t,\x)\\&=\sqrt{\e}\boldsymbol{\eta}, \ (t,\x)\in[0,T]\times{\1},\\
		\nabla\cdot \u^\e(t,\x)&=0, \ (t,\x)\in[0,T]\times{\1},\\
		\u^\e(0,\x)&=\u_0(\x), \ \x\in{\1},
	\end{aligned}
	\right.
\end{equation}which describe the motion of a viscous incompressible fluid in a domain ${\1}\subset\R^2$. Here, $\u^\e(t,\x)=(u_1^\e,u_2^\e)(t,\x)\in\R^2$ represents the velocity field at time $t$ and position $\x$, $p^\e(t,\x)\in\R$ denotes the pressure field, $\nu>0$ is the viscosity, $\u_0$ us the initial velocity and $\boldsymbol{\eta}$ is the stochastic forcing term. We consider our domain to be a torus, with ${\1}=[0,2\pi]^2$ and periodic boundary conditions in place.
\subsection{Vorticity formulation} In this article, we discuss the vorticity formulation of the model in \eqref{1.1} with two distinct types of noises. 
\subsubsection{Additive infinite-dimensional noise} 
Taking \textsl{curl} on both sides of the first equation of the system \eqref{1.1}, one  obtains the vorticity formulation, where the unknown is the vorticity $$	\xi^\e(t,\x)=\nabla^{\perp}\cdot\u^\e(t,\x)\equiv \partial_{x_1}u_2^{\e}(t,x)-\partial_{x_2}u_1^{\e}(t,x), \ (t,\x)\in[0,T]\times {\1}.$$ We know that the \textsl{curl} of a planar vector field has only one significant component $\xi^\e$, which is orthogonal to the plane. Therefore, with enough regularity of solutions,  the system \eqref{1.1} is equivalent to the following system: 
\begin{equation}\label{1.2}
	\left\{
	\begin{aligned}
		\frac{\partial \xi^\e}{\partial t}(t,\x)+\u^\e(t,\x)\cdot\nabla\xi^\e(t,\x)-&\nu\Delta\xi^\e(t,\x)\\&=\sqrt{\e}\W(\d t,\d \x), \ (t,\x)\in[0,T]\times{\1},\\
		\nabla\cdot \u^\e(t,\x)&=0, \ (t,\x)\in[0,T]\times{\1},\\
		\xi^\e(0,\x)&=\xi_0(\x), \ \x\in{\1},
	\end{aligned}
	\right.
\end{equation}with the periodic boundary conditions. For the sake of convenience, we are assuming $\nu=1$. Here, $\W(\cdot,\cdot)$ is the formal notation for some Gaussian perturbation defined on a filtered probability space $(\Omega,\mathscr{F},\{\mathscr{F}_t\}_{t\geq0},\P)$. The \textsl{Biot-Savart law} provides a relationship  between the velocity field $\u^\e$ and the vorticity $\xi^\e$, expressed as $\u^\e=\k*\xi^\e$, wherein $\k$ stands for the Biot-Savart kernel (see Subsection \ref{BSL}).

We interpret the system \eqref{1.2} in the sense of Walsh (see \cite{JBW}). Let $G(\cdot,\cdot,\cdot)$ be the fundamental solution to the heat equation on the flat torus ${\1}$. Then, the random field $\xi^\e=\{\xi(t,\x):(t,\x)\in[0,T]\times{\1}\}$ is a solution to the system \eqref{1.2} if it satisfies the following integral equation
\begin{align}\label{1.3}\nonumber
	\xi^\e(t,\x)&=\int_{\1} G(t,\x,\y)\xi_0(\y)\d\y+\int_0^t\int_{\1}\nabla_{\y} G(t-s,\x,\y) \cdot \u^\e(s,\y)\xi^\e(s,\y)\d\y\d s\\&\quad +\sqrt{\e}\int_0^t\int_{\1} G(t-s,\x,\y)\W(\d s,\d \y),\ \ \P\text{-a.s.},
\end{align}
for all $(t,\x)\in[0,T]\times{\1}$.

\subsubsection{Multiplicative finite-dimensional noise} Let us now assume that the stochastic perturbation be a finite-dimensional Wiener process with linear growth coefficient (see Hypothesis \ref{hyp1} below). 	With the above change in the noise term, one can rewrite the system \eqref{1.2} as follows:
\begin{equation}\label{FD1}
	\left\{
	\begin{aligned}
		\frac{\partial \xi^\e}{\partial t}(t,\x)&+\u^\e(t,\x)\cdot\nabla\xi^\e(t,\x)-\Delta\xi^\e(t,\x)\\&=\sqrt{\e}\sum_{j=1}^{n}\sigma_j(t,\x,\xi^\e(t,\x))\frac{\d}{\d t}\W^j(t), \ (t,\x)\in[0,T]\times{\1},\\
		\nabla\cdot \u^\e(t,\x)&=0, \ (t,\x)\in[0,T]\times{\1},\\
		\xi^\e(0,\x)&=\xi_0(\x), \ \x\in{\1},
	\end{aligned}
	\right.
\end{equation}where $\W(\cdot)$ is an $n$-dimensional Wiener process. 
We have a similar formulation to \eqref{1.3} for the solution to the system \eqref{FD1} in sense of Walsh (see \cite{JBW}) as
\begin{align}\label{FD2}\nonumber
	\xi^\e(t,\x)&=\int_{\1} G(t,\x,\y)\xi_0(\y)\d\y+\int_0^t\int_{\1}\nabla_{\y} G(t-s,\x,\y) \cdot \u^\e(s,\y)\xi^\e(s,\y)\d\y\d s\\&\quad +\sqrt{\e}\sum_{j=1}^{n}\int_0^t\int_{\1} G(t-s,\x,\y)\sigma_j(s,\y,\xi^\e(s,\y))\d \y\d\W^j(s), \ \ \P\text{-a.s.},
\end{align}
for all $(t,\x)\in[0,T]\times{\1}$.
\subsection{Literature review}
The analysis of stochastic partial differential equations (SPDEs) deals with the existence and uniqueness results, and various properties of solution processes such as the large deviation principle (LDP), uniform large deviation principle (ULDP), invariant measures, random attractors, and much more. In this work, we are able to formulate our system as a two-dimensional stochastic parabolic nonlinear equation with a distinct nonlinear term compared to most papers related to higher-dimensional spatial domains concerning stochastic parabolic nonlinear equations (cf. \cite{BFMZ,JLMSSS}, etc. and the references therein). We assume our problems in the vorticity form \eqref{1.2} and \eqref{FD1} by taking the \textsl{curl} of the system \eqref{1.1}, similarly to the approach used by \cite{BFMZ}, but with different types of noises. The main focus of this work is to analyze the asymptotic behavior of the solutions to the system \eqref{1.1} in the vorticity form with different types of stochastic forcing terms (as seen in \eqref{1.2} and \eqref{FD1}).

The intriguing research topic of large deviation theory in probability theory is examined to analyze the limiting behavior of rare event probabilities. The outstanding works by Donsker and Vardhan \cite{DV85,V66,V84}, as well as Freidlin and Wentzell \cite{MIFADW}, brought the large deviation theory to the forefront of probability theory research. Numerous authors, including \cite{HBAM,ZBSC,ZBBGTJ,ZBXPJZ,SCAD1,SCAD2,ICAM1,ZDRZ,WHSL,AKMTM7,AKMTM5,WL1,MRTZXZ,TXTZ,JZTZ}, have contributed to the field by establishing LDP for several types of SPDEs. Our goal in this article is to demonstrate the Freidlin-Wentzell Uniform Large Deviation Principle (FWULDP), also known as the ULDP from Freidlin and Wentzell. The \textsl{Contraction Principle}, as stated in Theorem 4.2.1 of \cite{ADOZ}, provides a way to address the LDP in the setting of additive Gaussian noise. This principle, which states that the LDP is preserved under continuous mappings, has been employed by the authors of \cite{ADOZ, MTMSBH, MTMTD}, and other papers, to obtain LDP for various types of SPDEs. Recently, the idea of \textsl{Contraction Principle} has been extended by the author in \cite{BW} to \textsl{Uniform Contraction Principle} in order to prove ULDP, considering uniformity across initial conditions that lie within a bounded subsets of an appropriate space.

We cannot apply the approach developed by \cite{ADPDVM} directly to establish the  ULDP for SPDEs that have non-compact sets of initial conditions, as this approach requires the convergence of initial conditions. For several applications, such as the study of exit time and place of a stochastic process $\eta$ from its domain, it is not necessarily required for the large deviations of $\eta$ to be uniform over bounded, non-compact subsets of the space (for more details, we refer to \cite{MS}).  Several authors have worked towards establishing ULDP for various types of SPDEs, with notable contributions from \cite{FCAM,AKMTM7,MS,MSLS,BW} and their cited references. The authors of \cite{MS} obtained a ULDP for a general class of SPDEs, but it was limited to the application on SPDEs defined in bounded domains with compact Sobolev embeddings. An LDP for a general class of   Banach space valued SPDEs that is uniform	  with respect to initial conditions in bounded subsets of the Banach space is established in \cite{MSABPD19}. Recently, the author in  \cite{BW} established the uniform contraction principle to obtain the ULDP for solutions of fractional stochastic reaction-diffusion equations, which is uniformly valid for initial conditions in bounded subsets of the $\L^2(\R^d)$-norm.

Under Hypothesis \ref{hyp1} (see below), the second main result of this article is the existence and uniqueness of a solution to the system \eqref{FD1} driven by a finite dimensional Wiener process with a linear growth coefficient. We adopted the idea from \cite{IG,IGDN,IGCR2,IGCR} for different types of SPDEs. The heat kernel estimates given by Lemma \ref{thrm2.1} and Lemma \ref{lem25} play a crucial role in our proof. We choose the method of localization to handle the non-Lipschitz nonlinear term in \eqref{1.2}, as suggested in \cite{IG,IGDN,IGCR,IGCR2,AKMTM7,AKMTM3}. First we show  that the truncated integral equation \eqref{FD7} has a unique global solution, which implies the local existence and uniqueness of the solution upto some stopping time $\tau$ to the integral equation \eqref{FD2}. We establish the global solvability results by showing that  $\tau=T$, $\P$-a.s. Furthermore, we assume our initial data $\xi_0\in\C({\1})$ and prove that the solution $\xi$ admits a modification which is a space-time continuous process, 

The final goal of this article is to establish ULDP for the solution of the integral equation \eqref{FD2}. Different definitions of ULDP can be shown to be equivalent with appropriate assumptions, as demonstrated in \cite{MS}. The authors in \cite{ABPFD} used a variational representation to establish the  FWULDP for a diverse family of infinite-dimensional stochastic dynamical systems. They have established the weak convergence approach for proving the ULDP with the help of \textsl{uniform Laplace principle} (ULP). They established the ULDP for the law of solution to certain SPDEs by  proving a sufficient condition, and to verify it one needs to explore more basic properties such as well-posedness and tightness of skeleton integral equations (\eqref{6.5} and \eqref{6.6} below), which are similar to the original integral equation  \eqref{FD2}. 

In a recent work \cite{MS}, the author established the equivalency between four different kind of definitions of ULDP and ULP.  He formulated the \textsl{Equicontinuous Uniform Laplace Principle} (EULP) and demonstrated it is equivalent to the FWULDP without any additional or compactness assumptions. He provided a sufficient condition for the proof of EULP under which the law of the solution follows ULDP with uniformity considered over initial conditions belonging to a bounded, but not necessarily compact, set. In the work \cite{MSLS}, the authors established FWULDP for the law of the solution to the stochastic Burgers type equation, assuming a polynomial nonlinearity of any order driven by a finite-dimensional Wiener process. The authors in \cite{LS23} derived the ULDP for the law of solutions to a class of semilinear SPDEs with bounded noise coefficient, which are driven by space-time white noise. Recently, the authors in \cite{AKMTM7} used the approach of \cite{MS} to demonstrate the UDLP for a stochastic generalized Burgers-Huxley equation that is driven by either a $\Q$-Wiener process with a linear growth of noise coefficient or a space-time white noise with a bounded noise coefficient. In the works \cite{AKMTM7, MSLS, LS23}, it is assumed that the uniformity with respect to the initial conditions that are in a bounded subsets does not need to be compact.
\subsection{Objectives and novelties of the paper}  This work aims to analyze the asymptotic behavior of the two-dimensional Navier-Stokes equations in the vorticity form with a stochastic forcing term (to be specified later). To the best of our understanding, this analysis presented in this article has not been previously investigated in available literature. This paper is structured into two parts, which focus on the noise structure:

In the first part, we assume that the equation is driven by an infinite-dimensional additive noise, which is white in time and colored in space.
\begin{itemize}
	\item 
	The noise structure of this part resembles that of the work \cite{BFMZ}. Minor modifications to the well-posed result presented in \cite{BFMZ} can be used to obtain the global solvability of the system \eqref{1.2}. Our objective  is to demonstrate that ULDP holds in two distinct topologies: $\C([0,T];\L^p({\1}))$ with $p>2$, where the uniformity is measured over initial conditions belonging to bounded subsets of $\L^p({\1})$, and $\C([0,T]\times{\1})$ with uniformity over initial conditions belonging to bounded subsets of $\C({\1})$. The proof of the ULDP in both cases are based on the uniform contraction principle established in \cite{BW}.
\end{itemize}
We assume in the second part that the equation is driven by a multiplicative finite-dimensional  Wiener process with a linear growth of the noise coefficient.
\begin{itemize}
	\item We use the approach from \cite{IG,IGCR2,IGCR,AKMTM3}, etc. to ensure the well-posedness of the system \eqref{FD1}. We prove the existence of a local solution to the integral equation \eqref{FD2} by utilizing a truncation and fixed point argument. With the help of stopping time arguments, tightness properties and Skorokhod's representation theorem, we establish the uniform global energy estimates of the local solution and consequently, the global existence of the solution.
	\item Our final goal is to verify the ULDP for the system \eqref{FD1}. {Given the initial conditions in a bounded subsets of an appropriate space, it is not sufficient to establish ULP to obtain UDLP without any additional assumption as discussed in \cite{MS}.} The motivation for the ULDP comes from the work of \cite{MS}, which provides a sufficient condition for the verification of EULP, leading to ULDP in two distinct topologies: $\C([0,T];\L^p({\1}))$ with $p>2$, where the uniformity is measured over initial conditions belonging to bounded subsets of $\L^p({\1})$, and $\C([0,T]\times{\1})$ with uniformity over initial conditions belonging to bounded subsets of $\C({\1})$. We use the embedding $\C([0,T]\times{\1})\hookrightarrow \C([0,T];\L^p({\1}))$, in the proof of ULDP in $\C([0,T]\times{\1})$ topology.
\end{itemize}

\subsection{Organization of the paper}The following structure is used to organize this article: We start Section \ref{Sec2} by introducing basic function spaces, heat kernel estimates, Biot-Savart law, and noise structure that assist us in formulating our problem. Then, we analyze the global solvability of the system \eqref{1.2}. Section \ref{LDP} begins by introducing the basics of LDP and ULDP, followed by recalling the uniform contraction principle (Proposition \ref{thrmU2}) and proceeding to the proof of local Lipschitz continuity of solutions of \eqref{L5} (Lemma \ref{lem4.8}). Subsequently, we present and demonstrate our main results of this section (Theorems \ref{thrmUL} and \ref{thrmUL1}), which are the ULDP for the law of the solutions of the system \eqref{1.2}. 
We analyze the perturbation caused by a finite-dimensional Wiener process in Sections \ref{FD} and \ref{UDLP}, as shown in the equation \eqref{FD1}.  We obtain global solvability of the integral equation \eqref{FD2} under Hypothesis \ref{hyp1} via the use of the truncation method and uniform energy estimates (Lemmas \ref{lem5.2} and \ref{lem5.4}, Proposition \ref{prop5.3}, and Theorem \ref{thrmex}). The ULDP for the law of the solutions to the integral equation \eqref{FD2} are presented in the final Section \ref{UDLP}. We first recall some useful notations and results from \cite{MS} including a sufficient condition \ref{cond} that, according to \cite[Theorem 2.10]{MS}, ensures EULP and hence ULDP. We establish intermediate results (Theorems \ref{thrm6.5} and \ref{thrm6.6}, Proposition \ref{prop6.9}, and Corollary \ref{cor6.10}) later that are useful to obtain our main results (Theorems \ref{thrm6.7} and \ref{thrm6.8}) in this section. We use uniform convergence in probability (Corollary \ref{cor6.11}) to prove Theorem \ref{thrm6.7}. We conclude the article with the proof of Theorem \ref{thrm6.8}, relying on Lemma \ref{lem6.12} for tightness and Theorem \ref{thrm6.13} for convergence in probability in supremum norm.

\section{Mathematical formulation}\label{Sec2}\setcounter{equation}{0}
We will begin this section by considering the vorticity form of the system \eqref{1.1}, and then discuss the required function spaces needed, along with the heat kernel and its associated estimates, Biot-Savrat law, and a brief overview of the noise structure which will lead us to our model and its solvability results.

\subsection{Function spaces} We denote by  $\x=(x_1,x_2)$ a point in $\R^2$ and  the scalar product and the norm in $\R^2$ by
$\x\cdot\y=x_1y_1+x_2y_2, \text{ and } \ |\x|=\sqrt{\x\cdot\x}, \  \x,\y\in\R^2,
$ respectively. Given $\z=\mathrm{Re}(\z)+i\mathrm{Im}(z)\in\mathbb{C}$, we denote the absolute value and complex conjugate of $\z$ by $|\z|$ and $\overline{\z}$, respectively, and 
$
|\z|=\sqrt{(\mathrm{Re}(\z))^2+(\mathrm{Im}(\z))^2}, \text{ and } \overline{\z}=\mathrm{Re}(\z)-i\mathrm{Im}(\z).
$ We define $\Z_0^2=\Z^2\backslash\{(0,0)\}$ and $\Z_+^2=\{\eta=(\eta_1,\eta_2)\in\Z^2:\eta_1>0\}\cup\{\eta=(0,\eta_2)\in\Z^2:\eta_2>0\}$. Let us denote the space of all complex-valued $2\pi$-periodic functions in $x_1$ and $x_2$ which are measurable and square integrable on ${\1}$ by $\L_\#^2({\1})$. The space $\L_\#^2({\1})$ is equipped with the scalar product and norm given by
\begin{align*}
	(f_1,f_2)_{\L^2({\1})}=\int_{\1}f_1(\x)\overline{f_2(x)}\d \x \  \text{ and } \  \|\cdot\|_{\L^2({\1})}^2=(\cdot,\cdot)_{\L^2({\1})},
\end{align*}respectively. We define the space $\mathbb{L}_{\#}^2({\1})=[\L_\#^2({\1})]^2$ consisting of all pairs $\u=(u_1,u_2)$ of complex-valued periodic functions equipped with the inner product 
\begin{align*}
	(\u,\v)_{\mathbb{L}_{\#}^2({\1})}:=\int_{\1} \u(\x)\cdot\overline{\v(\x)}\d \x= \int_{\1}\big[u_1(\x)\overline{v_1(\x)}+u_2(\x)\overline{v_2(\x)}\big]\d \x, \  \u,\v\in\mathbb{L}_{\#}^2({\1}).
\end{align*}An orthonormal basis for the space $\L_\#^2({\1})$ is given by $\{e_\eta\}_{\eta\in\Z^2}$, where 
\begin{align}\label{2.1}
	e_\eta(\x)=\frac{1}{2\pi} e^{i\eta\cdot\x}, \ \ \x\in{\1},  \text{ and } \ \eta\in\Z^2.
\end{align}
One can also consider non-zero mean value vectors in the same manner as discussed in \cite{RT}, even though mean value zero vectors are more commonly used and mathematically simpler to consider in periodic cases. We denote the space with  the zero-mean condition on the space $\L_\#^2({\1})$ by $\mathring{\L}_\#^2({\1})$.  An orthonormal system of eigenfunctions $\{e_\eta\}$ defined in \eqref{2.1}, with associated eigenvalues $\lambda_\eta=|\eta|^2$, is provided for the space $\mathring{\L}_\#^2({\1})$ by the operator $-\Delta$.

The real-valued functions in $\mathring{\L}_\#^2({\1})$ can be characterized by their Fourier series expansion as follows:
\begin{align*}
	\mathring{\L}_\#^2({\1})=\bigg\{g(\x)=\sum_{\eta\in\Z_0^2}g_\eta e_\eta(\x):\overline{g_\eta}=g_{-\eta}, 
	\ \text{ for any }\  \eta, \sum_{\eta\in\Z_0^2}|g_{\eta}|^2<\infty\bigg\}.
\end{align*}For every $p>2$, we denote the subspace of $\L^p({\1})$ consisting of zero mean and periodic scalar functions by $\mathring{\L}_\#^p({\1})$. Moreover, the spaces $\mathring{\L}_\#^p({\1})$ are Banach spaces with the norms inherited from the spaces $\L^p({\1})$.

The Laplacian operator $-\Delta$ with periodic boundary conditions is represented by $\A$. For every $a\in\R$, we define the powers of the operator $\A$ for $ g=\sum\limits_{\eta\in\Z_0^2}g_\eta e_\eta$ as
\begin{align*}
	\A^ag =\sum_{\eta\in\Z_0^2}|\eta|^{2a}g_{\eta}e_\eta,
\end{align*}and the domain of $\A^a$ is denoted by $\D(\A^a):=\bigg\{g=\sum\limits_{\eta\in\Z_0^2}g_\eta e_\eta: \sum\limits_{\eta\in\Z_0^2}|\eta|^{4a}|g_\eta|^2<\infty\bigg\}$. 

For any $a\in\R^+$ and $p\geq 1$, we define $\mathrm{H}^{a,p}=\{g\in \mathring{\L}_\#^p({\1}):\A^{\frac{a}{2}}g\in \mathring{\L}_\#^p({\1})\} $, as Banach spaces with the usual norm. For $p=2$, $\mathrm{H}^{a,2}$ become Hilbert spaces and we denote them by $\mathrm{H}^a$. For $a<0$, we define $\mathrm{H}^a$ as the dual space of $\mathrm{H}^{-a}$ with respect to the $\L^2$-scalar product.

We define the space regularity of periodic vector fields, zero mean value, and divergence-free through Laplace operator's corresponding action on each vector component. Therefore, we define the space $\H=\{\u\in \mathring{\mathbb{L}}_{\#}^2({\1}):\nabla\cdot\u=0\}$, here the divergence free condition has to be understood in the distributional sense. The scalar product from $[\L^2({\1})]^2$ allows $\H$ to be considered as a Hilbert space. The norm on the space $\H$ is given by $\|\cdot\|_\H$, $\|\u\|_{\H}=\sqrt{(\u,\u)_\H}$. Let $\left\{\frac{\eta^{\perp}}{|\eta|}e_\eta\right\}_{\eta\in\Z_0^2}$ be the basis of the space $\H$, where $\eta^{\perp}=(-\eta_2,\eta_1)$ and $e_\eta$ is defined in \eqref{2.1}. For $p>2$, set $\LL_{\sigma}^p(\1):= \mathbb{H}\cap\mathbb{L}_{\#}^p(\1)$, which forms Banach spaces with the norms inherited from the space $\mathbb{L}_{\#}^2(\1)$. Similarly,  we set the vector space,  $\H^{a,p}=\{\u\in\LL_{\sigma}^p({\1}):\A^{\frac{a}{2}}\u\in\LL_{\sigma}^p({\1})\}$. These are Banach spaces with the usual norm, and for $p=2$ they become Hilbert spaces and we denote them by $\H^a$. For $a<0$, we define $\H^a$ as the dual space of $\H^{-a}$ with respect to the $\H$-scalar product. Moreover, the Poincar\'e inequality holds. The zero mean value assumption provides that the norms $\|\u\|_{\H^{a,p}}=\|\A^{\frac{a}{2}}\u\|_{\mathbb{L}^p}$ and $\big(\|\u\|_{\LL_{\sigma}^p}^p+\|\u\|_{\H^{a,p}}^p\big)^\frac{1}{p}$ are equivalent. 

We use the following Rellich-Kondrachov compactness Theorem  in the sequel (see \cite[Theorem 9.16]{HB}):
\begin{enumerate}
	\item For every $p\in(2,\infty)$, the space $\H^{1,p}$ is compactly embedded in $\LL_{\sigma}^\infty({\1})$, that is, there exists a constant $C_p$ such that 
	\begin{align}\label{2.2}
		\|\u\|_{\LL_{\sigma}^\infty} \leq C_p\|\u\|_{\H^{1,p}}.
	\end{align}
	\item The space $\mathrm{H}^a$ is compactly embedded in $\L^\infty({\1})$ for $a>1$.
\end{enumerate}
Let $(\U,\|\cdot\|_{\U})$ and $(\V,\|\cdot\|_{\V})$ be any two normed vector spaces. The space of all linear bounded operators from $\U$ into $\V$ is denoted by $\mathcal{L}(\U,\V)$. We use the notation $(\cdot,\cdot)$ and $\langle \cdot,\cdot\rangle_{\U'\times\U}$ for scalar product and duality pairing between $\U'$ and $\U$, respectively.

\subsection{The heat kernel}
We interpreted system \eqref{1.2} in the sense of Walsh (see \cite{JBW}) as expressed in \eqref{1.3}, with the heat kernel $G$. We need to make accurate estimations of $G$ to complete our calculations, as it is a critical factor.

The operator $\A$ generates a semigroup $\S(t)=e^{-t\A}$, for $\xi\in \mathring{\L}_{\#}^2(\1)$ and $t\in[0,T]$, we have 
\begin{align}\label{2.3}
	[\S(t)\xi](\x)=\sum_{\eta\in\Z^2} e^{-t|\eta|^2}(\xi,e_\eta)_{\L^2}e_\eta(\x)=\frac{1}{2\pi}\sum_{\eta\in\Z^2}(\xi,e_\eta)_{\L^2}e^{-t|\eta|^2+i\eta\cdot\x}.
\end{align}Moreover, the action of the semigroup $\S(\cdot)$ on the vorticity $\xi$ can be represented by the convolution
\begin{align}\label{2.4}
	[\S(t)\xi](\x)=\int_{\1}G(t,\x,\y)\xi(\y)\d y,
\end{align}where $G$ is the heat kernel of the following problem
\begin{equation}\label{2.5}
	\left\{
	\begin{aligned}
		\frac{\partial v}{\partial t}(t,\x)-\Delta v(t,\x)&=0, \ \ (t,\x)\in(0,T]\times{\1},\\
		v(t,\cdot) &\text{ is periodic }, \ \ t\in[0,T],\\
		v(0,\x)&=\delta_0(\x-\y), \ \ \x,\y\in{\1},
	\end{aligned}
	\right.
\end{equation}
where $\delta_0(\cdot)$ is the Dirac delta distribution centered at $(0,0)$. Using Fourier series expansion, we obtain the following explicity form of heat kernel $G$
\begin{align}\label{2.6}
	G(t,\x,\y)=\frac{1}{(2\pi)^2}\sum_{\eta\in\Z^2}e^{-t|\eta|^2+i\eta\cdot(\x-\y)}.
\end{align}For our calculations, we need another form of heat kernel which can be derived using the method of images (see \cite[Chapter 2.7, Section 5 or Chapter 2.11, Section 3]{HDHPM} or \cite[Chapter 7, Section 2]{EMS}) 
\begin{align}\label{2.7}
	G(t,\x,\y)=\frac{1}{4\pi t}\sum_{\eta\in\Z^2}e^{-\frac{|\x-\y+2\eta\pi|^2}{4 t}}.
\end{align}From \eqref{2.6} or \eqref{2.7}, one can observe the following properties of the heat kernel $G$ for all $(t,\x,\y)\in [0,T]\times{\1}\times{\1}$:
\begin{enumerate}
	\item \textsl{Symmetry:} $G(t,\x,\y)=G(t,\y,\x)$, 
	\item $G(t,\x,\y)=G(t,0,\x-\y)$.
\end{enumerate}Let us recall a result from \cite{BFMZ} which gives useful estimates on the heat kernel and its gradient in the two-dimensional case.
\begin{lemma}[{\cite[Theorem 4]{BFMZ}}]\label{thrm2.1}
	For fixed $s\in(0,t)$ and $\x\in\T^2,$ the following estimates hold:
	\begin{enumerate}
		\item For $\beta\in(0,4/3)$, there exists a positive constant $C_\beta$ such that
		\begin{align}\label{2.8}
			\int_{\1}|\nabla_{\y}G(s,\x,\y)|^\beta\d \y&\leq C_\beta s^{1-\frac{3\beta}{2}},\\\label{2.9}
			\int_0^t	\int_{\1}|\nabla_{\y}G(s,\x,\y)|^\beta\d \y\d s&\leq C_\beta t^{2-\frac{3\beta}{2}}.
		\end{align}
		\item For every $\beta\in(0,2)$, there exists a positive constant $C_\beta$ such that
		\begin{align}\label{2.10}
			\int_{\1}|G(s,\x,\y)|^\beta\d \y&\leq C_\beta s^{1-\beta},\\\label{2.11}
			\int_0^t	\int_{\1}|G(s,\x,\y)|^\beta\d \y\d s&\leq C_\beta t^{2-\beta}.
		\end{align}
	\end{enumerate}
\end{lemma}

The following result shows the regularizing effect of convolution with the gradient of the heat kernel $G$. Let $J$ be the linear operator defined as 
\begin{align}\label{R1}(J\boldsymbol{\phi})(t,\x)=\int_0^t\int_{\1}\nabla_{\y}G(t-s,\x,\y)\cdot\boldsymbol{\phi}(s,\y)\d \y\d s, \ \ \text{ for } \ \ (t,\x)\in[0,T]\times{\1},
\end{align} and $\boldsymbol{\phi}$ smooth enough. We have the following result from \cite{BFMZ}. 
\begin{lemma}[{\cite[Lemma 6]{BFMZ}}]\label{lem25}
	The following hold:
	\begin{enumerate}
		\item 
		Let $p\geq 1,\ \alpha\geq 1,\ \beta\in\left[1,\frac{4}{3}\right), \ \gamma>\frac{2\beta}{2-\beta}$ such that $\frac{1}{\beta}=1+\frac{1}{p}-\frac{1}{\alpha}$. Then, $J$ is a bounded linear operator from $\L^\gamma(0,T;\LL_{\sigma}^\alpha(\1))$ into $\L^\infty(0,T;\L^p(\1))$. Moreover, there exists a positive constant $C_\beta$ such that 
		\begin{align}\label{R2}\nonumber
			\|(J\boldsymbol{\phi})(t)\|_{\L^p}&\leq C_\beta\int_0^t(t-s)^{\frac{1}{\beta}-\frac{3}{2}}\|\boldsymbol{\phi}(s)\|_{\LL_{\sigma}^\alpha}\d s\\&\leq C_\beta t^{\frac{1}{\beta}-\frac{3}{2}+\frac{\gamma-1}{\gamma}}\bigg(\int_0^t\|\boldsymbol{\phi}(s)\|_{\LL_{\sigma}^\alpha}^\gamma\d s\bigg)^{\frac{1}{\gamma}},
		\end{align}for all $t\in[0,T]$.
		\item Let $p>4$ and $\gamma>\frac{2p}{p-2}$. Then, $J$ is a bounded linear operator from\\ $\L^\gamma(0,T;\LL_{\sigma}^p(\1))$ into $\C([0,T]\times{\1})$. Moreover, there exists a positive constant $C_{p,T}$ such that 
		\begin{align}\label{R3}
			\sup_{t\in[0,T]}\sup_{\x\in{\1}}|(J\boldsymbol{\phi})(t,\x)|\leq C_{p,T}\bigg(\int_0^T\|\boldsymbol{\phi}(s)\|_{\LL_{\sigma}^p}^\gamma\d s\bigg)^\frac{1}{\gamma}.
		\end{align} 
	\end{enumerate}
\end{lemma}

\subsection{The Biot-Savart law}\label{BSL} Let us discuss the Biot-Savart law, which provides an explicit expression of the velocity field $\u^\e$ in terms of the vorticity $\xi^\e$, as described in \cite[Section 2.4.1]{AJMALB} and \cite[Section 1.2]{CMMP}. Taking \textsl{curl} on both side of the relation $\xi^\e=\nabla^{\perp}\cdot\u^\e$, we find
\begin{equation}\label{2.12}
	\left\{
	\begin{aligned}
		-\Delta\u^\e&=\nabla^{\perp}\xi^\e,\\
		\nabla\cdot\u^\e&=0,\\
		\u^\e &\text{ is periodic}.
	\end{aligned}
	\right.
\end{equation}The above system allows us to express velocity field $\u^\e$ in terms of vorticity $\xi^\e$. In terms of the Fourier series, if $\xi^\e(\x)=\frac{1}{2\pi}\sum\limits_{\eta\in\Z_0^2}\xi_\eta^\e e^{i\eta\cdot\x}$, then 
\begin{align}\label{2.13}
	\u^\e(\x)=-\frac{i}{2\pi}\sum_{\eta\in\Z_0^2} \frac{\eta^{\perp}}{|\eta|^2}\xi_\eta^\e e^{i\eta\cdot\x}.
\end{align}It is clear from the above relation that the velocity field $\u^\e$ has one order more of regularity in view of vorticity field $\xi^\e$, that is, if $\xi^\e \in\mathrm{H}^{a-1,p}$ then $\u^\e\in\H^{a,p}$. In particular, the norms $\|\u^\e\|_{\H^{a,p}}$ and $\|\xi^\e\|_{\mathrm{H}^{a-1,p}}$ are equivalent.

From \cite[Section 1.2]{CMMP}, we know that the Biot-Savart law allows us to write the velocity field $\u^\e$ in terms of the vorticity $\xi^\e$ as 
\begin{align}\label{2.14}
	\u^\e(\x)=(\k*\xi^\e)(x)=\int_{\1}\k(\x-\y)\xi^\e(\y)\d \y,
\end{align}where $\k$ is the Biot-Savart kernel given by 
\begin{align}\label{2.15}
	\k=\nabla^\perp G=\left(-\frac{\partial G}{\partial x_2},\frac{\partial G}{\partial x_1}\right),
\end{align}and $G$ is the Green function of $-\Delta$ on the torus with mean zero. 

Let us recall some basic properties of the Biot-Savart kernel from \cite{ZBFFMM}.
\begin{lemma}[{\cite[Lemma 2.17]{ZBFFMM}}]\label{lem2.2}
	For every $p\in[1,2)$, the map $\k$ defined in \eqref{2.15} is an $\mathbb{L}^p(\1)$ divergence free (in the distributional sense) vector field.
\end{lemma}
We have some useful estimates which play a crucial role. Using \eqref{2.13}, the Sobolev embedding $\H^{1,p}\subset\LL_{\sigma}^\infty(\1)$ for $p>2$, and the equivalency of the norms $\|\u^\e\|_{\H^{1,p}}$ and $\|\xi^\e\|_{\L^p}$, we obtain for any $p>2$
\begin{align}\label{2.16}
	\|\k*\xi^\e\|_{\LL_{\sigma}^\infty}=\|\u^\e\|_{\LL_{\sigma}^\infty}\leq C_p\|\u^{\e}\|_{\H^{1,p}}\leq C_p\|\xi^\e\|_{\L^p},
\end{align}for some positive constant $C_p$.

Using \eqref{2.14}, Lemma \ref{lem2.2}, Young's inequality for convolution (with exponent $p\geq 1, \alpha\in[1,2)$, and $ \beta\geq 1$, and $1+\frac{1}{p}=\frac{1}{\alpha}+\frac{1}{\beta}$), we obtain 
\begin{align}\label{2.17}
	\|\k*\xi^\e\|_{\LL_{\sigma}^p}=\|\u^\e\|_{\LL_{\sigma}^p}\leq \|\k\|_{\LL_{\sigma}^\alpha}\|\xi^\e\|_{\L^\beta}.
\end{align}
\subsection{A short note on stochastic forcing term}\label{short}
Let us discuss the stochastic term appearing in the system \eqref{1.2}. Let $(\Omega,\mathscr{F},\{\mathscr{F}_t\}_{t\geq0},\P)$ be a given stochastic basis.
For any given $T>0$, let $\Q:\mathring{\L}_\#^2(\1)\to\mathring{\L}_\#^2(\1)$ be a positive symmetric bounded linear operator. Set $\L_\Q^2(\1)$ as the completion of the space of all square integrable, zero mean-value, periodic functions $\phi:{\1}\to\R$  with respect to the scalar product $(\phi,\psi)_{\L_\Q^2}=(\Q\phi,\psi)_{\L^2}$, for all $\phi,\psi\in\L_\Q^2(\1)$. 

Let $\mathcal{H}_T=\L^2(0,T;\L_\Q^2(\1))$ denote the real separable Hilbert space, with the scalar product \begin{align*}
	(g_1,g_2)_{\mathcal{H}_T}=\int_0^T(g_1(t),g_2(t))_{\L_\Q^2}\d t=\int_0^T(\Q g_1(t),g_2(t))_{\L^2}\d t.
\end{align*} Let us now consider the isonormal Gaussian process $\W=\{\W(h);h\in\mathcal{H}_T\}$ (for more details see \cite{DN}). The map $h\mapsto\W(h)$ provides a linear isometry from $\mathcal{H}_T$ onto $\mathcal{H}$, which is a closed subset of $\L^2(\Omega,\mathscr{F},\P)$, consisted of zero-mean Gaussian random variables and the isometry can be represented as 
\begin{align*}
	\E[\W(h)\W(g)]=(h,g)_{\mathcal{H}_T}.
\end{align*} The stochastic term appearing in \eqref{1.2} should be understand in the following sense, for $h\in\mathcal{H}_T$, we set 
\begin{align}\label{2.18}
	\W(h)=\int_0^T\int_{\1}h(s,\y)\W(\d s,\d\y),
\end{align}namely, the above random variable is a zero-mean Gaussian random variable with the covariance $\E[(\W(h))^2]=\|h\|_{\mathcal{H}_T}^2$.

The stochastic term above, expressed through linear isometry $\W$, can be interpreted in the framework of Walsh (as detailed in \cite{JBW} or \cite{DaZ}). One can express $\W(h)$ as 
\begin{align}\label{2.19}
	\W(h)=\sum_i\int_0^T(h(s),\tilde{e}_i)_{\L_\Q^2}\d\beta_s(\tilde{e}_i),
\end{align}where $\{\tilde{e}_i\}$ is a complete orthonormal basis of $\L_\Q^2(\1)$ and $\beta_s(\tilde{e}_i)=\W(\chi_{[0,s]}\tilde{e}_i)$, hence $\{\beta(\tilde{e}_i)\}$ is a sequence of independent standard 1-dimensional Brownian motions on the probability space $(\Omega,\mathscr{F},\P)$ adapted to the filtration $\{\mathscr{F}_t\}_{t\geq0}$. Setting $\M_t(A):=\W(\chi_{[0,t]}\chi_{A})$ for all $t\in[0,T]$ and $A\in\mathcal{B}_b(\R^2)$, we construct a martingale measure with the covariance operator $\Q$ and \eqref{2.19} coincides with the stochastic integral in the sense of Walsh (\cite{JBW}). Moreover, the isonormal Gaussian process $\W$ can be associated to a $\Q$-Wiener process $\mathscr{W}_t$ on $\mathring{\L}_\#^2(\1)$ (see \cite{DaZ}) in the following way
\begin{align}\label{2.20}
	(\mathscr{W}_t,h)_{\L^2}=\W(\chi_{[0,t]}h), \ \ \text{ for all } \ \ h\in\mathring{\L}_\#^2(\1),
\end{align}and \eqref{2.18} coincides with the integral with respect to $\mathscr{W}$, in a sense made precise in \cite[Section 3.4]{RCDLQS}. The stochastic convolution term appearing on the right-hand side of \eqref{1.3} is meaningful in the ways described above. The stochastic forcing term is periodic with zero-mean in the space variable. It is not unexpected that we need the covariance operator $\Q$ to have some regularizing effect when working in a two-dimensional space, as stated in \cite[Pages 144-146]{DaZ}. To address this issue, one can consider using the covariance operator $\Q$ with an appropriate regularizing effect. For example, one can choose a covariance operator of the  form 
\begin{align}\label{2.21}
	\Q=(-\Delta)^{-a}, \ \ \text{ for some} \ \ a>0.
\end{align}Thus, we have 
$
\Q e_\eta=\frac{1}{|\eta|^{2a}}e_\eta, \ \text{for all} \ \eta\in\Z_0^2,
$ and $\{\tilde{e}_\eta\}$ is a complete orthonormal basis of $\L^2_\Q(\1)$ defined as $\tilde{e}_\eta(\x)=\frac{1}{\sqrt{2}\pi}|\eta|^a\cos(\eta\cdot\x)$ and $\tilde{e}_{-\eta}(\x)=\frac{1}{\sqrt{2}\pi}|\eta|^a\sin(\eta\cdot\x)$, for $\eta\in\Z_+^2$. 
From \cite[Eq$^n$ (2.24)]{BFMZ}, we know that the stochastic term $\int_0^t\int_{\1}\G(t-s,x,y)\W(\d s,\d y)$ is well-defined for $a>0$, which is equivalent to have $G(t-\cdot,x,\cdot)\in\mathcal{H}_t$, for every $t>0$. It can justified as
\begin{align*}
	\|G(t-\cdot,x,\cdot)\|_{\mathcal{H}_t}^2\leq \frac{1}{8\pi^2}\sum_{\eta\in\Z_0^2}\frac{1}{|\eta|^{2+2a}},
\end{align*}the above series is convergent if and only if $a>0$. Thus, the hypothesis $a>0$ is sufficient for space-time continuity of the stochastic convolution's trajectories (see \cite[Theorem 2.13]{GDP1}).  The noise formulation of our work is similar to \cite{BFMZ}, where the author considered two-dimensional SNSE \eqref{1.1} in the vorticity form \eqref{1.2} with $\e=1$, and established the existence and uniqueness of the solution $\xi$ of \eqref{1.2} in the spaces $\C([0,T];\L^p(\1)),$ for $p>2$ and $\C([0,T]\times\1)$ whenever the initial data $\xi_0\in\L^p(\1)$ and $\xi\in\C(\1)$, respectively.

\subsection{Solvability results of the system \eqref{1.2}}\label{Sec3}
The main focus of this subsection is to recall the existence and uniqueness of the solution to the system \eqref{1.2}. We use the Biot-Savart law, which gives the relation  $\u^\e=\k*\xi^\e$. Set $\q(\xi^\e)=\xi^\e(\k*\xi^\e)$,  that is,
\begin{align}\label{S1}
[\q(\xi^\e)](\x)=\xi^\e(\x)\int_{\1}\k(\x-\y)\xi^\e(\y)\d \y.
\end{align}Using H\"older's inequality and \eqref{2.16}, we obtain for any $p>2$
\begin{align}\label{S2}
\|\q(\xi^\e)\|_{\LL_{\sigma}^p}\leq \|\xi^\e\|_{\L^p}\|\k*\xi^\e\|_{\LL_{\sigma}^\infty}\leq C_p\|\xi^\e\|_{\L^p}^2,
\end{align}which implies that $\q:\L^p(\1)\to\LL_{\sigma}^p(\1)$ for any $p>2$. We can write the system  \eqref{1.2} in an equivalent form, where the velocity field $\u^\e$ would not appear. That is, 
\begin{align}\label{S5}\nonumber
\xi^\e(t,\x)&=\int_{\1} G(t,\x,\y)\xi_0(\y)\d\y+\int_0^t\int_{\1}\nabla_{\y} G(t-s,\x,\y) \cdot \q(\xi^\e(s,\cdot))(\y)\d\y\d s\\&\quad +\sqrt{\e}\int_0^t\int_{\1} G(t-s,\x,\y)\W(\d s,\d y), \ \ \P\text{-a.s.}
\end{align}
for all $t\in[0,T]$. 
Let us provide the definition of weak solution (in analytic sense) to the system \eqref{1.2}. 

\begin{definition}
We say that an ${\L}^p$-valued ($p>2$) continuous $\mathscr{F}_t$-adapted stochastic process $\xi^\e$ is a \textsl{weak solution} to the system \eqref{1.2}, if it solves the following, for every $t\in[0,T],\ \phi\in\mathrm{H}^b$ with $b>2,$ we have
\begin{align}\label{S3}\nonumber
	\int_{\1} \xi^\e(t,\y)\phi(\y)\d \y&-\int_0^t\int_{\1}\xi^\e(s,\y)\Delta\phi(\y)\d \x\d s-\int_0^t\int_{\1}\q(\xi^\e(s,\cdot))(\y)\cdot\nabla\phi(\y)\d \y\d s\\&=\int_{\1}\xi_0(\y)\phi(\y)\d \y+\sqrt{\e}\int_0^t\int_{\1}\phi(\y)\W(\d s,\d \y),   \ \ \P\text{-a.s.}
\end{align}
\end{definition}The third term appearing in the left hand side of \eqref{S3} is well-defined, which can be justified as follows: for $s>1$
\begin{align}\label{S4}\nonumber
\bigg|\int_{\1}\q(\xi^\e(s,\cdot))(\y)\cdot\nabla\phi(\y)\d \y\bigg|&\leq \|\nabla\phi\|_{\L^\infty}\|\q(\xi^\e(s,\cdot))\|_{\LL_\sigma^1}\\&\nonumber\leq C\|\nabla\phi\|_{\mathrm{H}^s}\|\k*\xi^\e(s,\cdot)\|_{\LL_{\sigma}^2}\|\xi^\e(s,\cdot)\|_{\L^2}\\&\leq C\|\phi\|_{\mathrm{H}^{s+1}}\|\xi^\e(s,\cdot)\|_{\L^2}^2,
\end{align}where we have used \eqref{2.17}, H\"older's and Sobolev's inequalities.

By the equivalency of solution concepts given in \cite[Proposition 3.5]{IG}, we consider the solution to the system \eqref{1.2} according to Walsh (see \cite{JBW}) being provided by the integral equation \eqref{1.3}.

Let us recall the the existence and uniqueness result for the solution to the system \eqref{1.2} from \cite{BFMZ}.
\begin{proposition}[{\cite[Theorem 1]{BFMZ}}]
Let us assume that $\xi_0\in\L^p({\1})$, for $p>2$, and  $a>0$ in \eqref{2.21}. Then, there exists a unique $\L^p$-valued $\mathscr{F}_t$-adapted continuous process $\xi^\e$ satisfying the integral equation \eqref{1.3}. Moreover, if $\xi_0\in\C({\1})$, then the solution admits a modification which is a space-time continuous process.
\end{proposition}

\section{Uniform Large Deviation Principle: Additive Noise}\label{LDP}\setcounter{equation}{0}
In this section, we first recall the basics of  large deviation theory. Then, we discuss the concept of uniform large deviation principle (ULDP) and establish it for the law of the solution of the system \eqref{1.2}. 
\subsection{Large deviation principle}Let $\EE$ denote a  Polish space (complete separable metric space) with the Borel $\sigma$-field $\mathcal{B}(\EE)$. Let $\{\mu^\e\}_{\e>0}$ be a family of probability measures on $(\EE,\mathcal{B}(\EE))$. We provide some basic definitions related to large deviation principle.
\begin{definition}[Rate function]
A function $\I:\EE\to[0,\infty]$ is called a \textsl{rate function} on $\EE$ if it is a lower semicontinuous in $\EE$. A rate function $\I(\cdot)$ on $\EE$ is said to be a \textsl{good rate function} on $\EE$ if for every $M\in[0,\infty)$, the level set $K:=\{x\in\EE:\I(x)\leq M\}$ is a compact subset of $\EE$.
\end{definition}
\begin{definition}[Large deviation principle]
Let $\I:\EE\to[0,\infty]$ be a rate function on a Polish space $\EE$. The sequence $\{\mu^{\e}\}_{\e>0}$ satisfies the LDP on the space $\EE$ with the rate function $\I(\cdot)$, if the following hold:
\begin{enumerate}
	\item For any closed subset $\F\subset\EE$,
	\begin{align*}
		\limsup_{\e\to0}\e\log\mu^{\e}(\F)\leq -\inf_{x\in\F}\I(x).
	\end{align*}	\item For any open subset $\mathrm{O}\subset\EE$, 
	\begin{align*}
		\liminf_{\e\to0}\e\log\mu^{\e}(\mathrm{O}) \geq -\inf_{x\in\mathrm{O}}\I(x).
	\end{align*}
\end{enumerate}
\end{definition}
Let us now discuss the LDP for $\EE$-valued random variables. Let $(\Omega,\mathscr{F},\{\mathscr{F}_t\}_{t\geq0},\P)$ be a complete probability space satisfying  the usual conditions. 

Given $\e>0$, let $\mathscr{G}^\e:\EE_0\times \C([0,T]\times\1)\to\EE$, (where $\EE_0$ is an another Polish space) be a measurable map and set 
\begin{align}\label{U1}
\mathrm{Y}^\e=\mathscr{G}^\e(\W)	, \ \ \text{ for all } \e>0.
\end{align} Then, $\{\mathrm{Y}^\e\}_{\e>0}$ is a family of $\EE$-valued random variables. Let $\{\mu^\e\}_{\e>0}$ be the the law of $\{\mathrm{Y}^\e\}_{\e>0}$ on the space $\EE$. Then, the family $\{\mathrm{Y}^\e\}_{\e>0}$ is said to satisfy the LDP on the space $\EE$  if the family of measures $\{\mu^\e\}_{\e>0}$ satisfies the LDP on the space $\EE$.

For any given $M>0$, define the following sets:
\begin{align}\label{U2}
\mathfrak{U}^M&:=\bigg\{\phi\in\L^2([0,T]\times\1): \int_0^T\int_{\1}|v(s,\y)|^2\d \y\d s\leq M\bigg\},\\\label{U3}
\mathscr{P}_2&:=\bigg\{v: \int_0^T\int_{\1}|v(s,\y)|^2\d \y\d s<\infty, \ \P\text{-a.s.}\bigg\},\\\label{U4}
\mathscr{P}_2^M&:=\bigg\{v\in\mathscr{P}_2: \phi(\omega)\in\mathfrak{U}^M, \ \P\text{-a.s.}\bigg\}.
\end{align}Note that $\mathfrak{U}_M$ is a compact metric space equipped with the weak topology of $\L^2([0,T]\times\1)$. 

For every $\psi\in\EE$, we define a map $\I(\cdot):\EE\to [0,\infty]$ by 
\begin{align}\label{U5}
\I(\psi):=\inf_{\left\{v\in \L^2([0,T]\times\1):\psi= \mathscr{G}^0\left(\int_0^{\cdot}\int_{\1}v(s,\y)\d \y\d s\right)\right\}}\bigg\{\frac{1}{2}&\int_0^T\int_{\1}|v(s,\y)|^2\d \y\d s\bigg\}, 
\end{align}with the convention that the infimum over the empty set is infinite. 
We consider the following linear stochastic equation:
\begin{equation}\label{U6}
\left\{
\begin{aligned}
	\d \zeta^\e(t,\x)+\A\zeta^\e(t,\x)\d t&=\sqrt{\e}\W(\d t,\d\x),\ \ (t,x)\in(0,T)\times\mathbb{T}^2,\\
	\zeta^\e(0,\x)&=0.
\end{aligned}
\right.
\end{equation} 
The solution $	\zeta^{\e}(t,\x)=\sqrt{\e}\int_0^t\int_{\1}G(t-s,\x,\y)\W(\d s,\d\y), \  (t,\x)\in(0,T]\times{\1}$ is a stochastic convolution that is Gaussian and independent of the initial condition $\xi_0$. 
From \cite[Theorem 12.11]{DaZ}, we know that the LDP holds in $\C([0,T];\L^2(\1))$ for the laws of solutions of  the linear stochastic system \eqref{U6}.  Note that $\L^p(\1)$, for $p>2$ is  a Banach space continuously, and as a Borel set, embedded into $\L^2(\1)$. Then also the spaces $\C([0, T ]; \L^p(\1))$ and $\L^2(0, T ;\L^p(\1))$ are continuously, and as a Borel set, embedded into $\L^2(0,T;\L^2(\1))$. Taking into account of \cite[Theorem 12.16]{DaZ}, one can obtain the LDP holds in $\C([0,T];\L^p(\1))$, for $p>2$.  Using similar arguments one can conclude that the LDP holds in $\C([0,T]\times\1)$.

\subsection{Uniform large deviation principle}\label{SULDP}In this subsection, we recall the definition ULDP, to be more specific, Freidlin-Wentzell uniform large deviation principle (FWULDP), see \cite{MIFADW} from the work \cite{MS}, and a uniform contraction principle for proving  ULDP in a separable Banach space from \cite{BW}. 

Let $(\EE,d)$ be a Polish space and $\EE_0$ be a set. For any $\xi_0\in\EE_0$, define a rate function, $\I_{\xi_0}:\EE\to[0,\infty]$ (cf. \cite{ABPFD,MS,LS23}). Let $\Lambda_{\xi_0}(s_0)$ denote the level set
\begin{align*} 
\Lambda_{\xi_0}(s_0):=\big\{f\in\EE:\I_{\xi_0}(\psi)\leq s_0\big\},
\end{align*} 
for $s_0\in[0,\infty]$. Furthermore, let $\mathscr{T}$ be the collection of subsets of  $\EE_0$ and $b(\e)$ be a function which converges to $0$ as $\e\to0$.  

If $g:=g(t,x)$ is a random field and $\EE$ is a function space (will be fixed later), then $g$-a.s. in $\EE$ means that $g$ has a stochastic version in $\EE$, $\mathbb{P}$-a.s. For any metric space $(\EE,d)$, the distance between an element $h\in\EE$ and a set $\B\subset \EE$ is defined by 
\begin{align*}
\mathrm{dist}_{\EE}(g,\B):=\inf_{h\in\B}d (g,h).
\end{align*}

\begin{definition}[FWULDP]\label{def6.1}
A family of $\EE$-valued random variables $\{\mathrm{Y}_{\xi_0}^\e\}_{\e>0}$ is said to satisfy a \textsl{Freidlin-Wentzell uniform large deviation principle} with the speed $b(\e)$ and the rate function $\I_{\xi_0}(\cdot)$, uniformly over $\mathscr{T}$ if the following hold:
\begin{enumerate}
	\item For any $\mathrm{E}\in\mathscr{T},\ t_0\geq 0,$ and $\delta>0$
	\begin{align*}
		\liminf_{\e\to0}\inf_{\xi_0\in\mathrm{E}}\inf_{\psi\in\Lambda_{\xi_0}(t_0)}\bigg\{b(\e)\log\mu^{\e}\big(d (\mathrm{Y}_{\xi_0}^\e,\psi)<\delta\big)+\I_{\xi_0}(\psi)\bigg\}\geq 0.
	\end{align*}\item For any $\mathrm{E}\in\mathscr{T},\ t_0\geq 0,$ and $\delta>0$
	\begin{align*}
		\limsup_{\e\to0}\sup_{\xi_0\in\mathrm{E}}\sup_{s\in[0,t_0]}\bigg\{b(\e)\log\mu^{\e}\big(\mathrm{dist}_{\EE} (\mathrm{Y}_{\xi_0}^\e,\Lambda_{\xi_0}(s))\geq\delta\big)+s\bigg\}\leq 0.
	\end{align*}
\end{enumerate}	
\end{definition}
\begin{proposition}[{\cite[Uniform contraction principle, Theorem 2.5]{BW}}]\label{thrmU2}
Let $\mathscr{T}$ be a nonempty set and $\mathrm{X}_1,\mathrm{X}_2$ be any two separable Banach spaces. Further assume that the family $\{\nu^\e\}_{\e>0}$ of probability measures on $(\mathrm{X}_1,\mathcal{B}(\mathrm{X}_1))$, satisfies an LDP with the good rate function $\I:\mathrm{X}_1\to[0,\infty]$. For any given $\xi_0\in\mathscr{T}$, let $\S_{\xi_0}:\mathrm{X}_1\to\mathrm{X}_2$ be a locally Lipschitz continuous mapping in the sense that for every $R>0$, there exists a positive constant $L_R$, such that for all ${\xi_0}\in\mathscr{T}, x_1,x_2\in\mathrm{X}_1$ with $\|x_1\|_{\mathrm{X_1}}\leq R$ and $\|x_1\|_{\mathrm{X_2}}\leq R$,
\begin{align*}
	\|\S_{\xi_0}(x_1)-\S_{\xi_0}(x_2)\|_{\mathrm{X_2}}\leq L_R\|x_1-x_2\|_{\mathrm{X_1}}.
\end{align*}For any given ${\xi_0}\in\mathscr{T}$ and $\e>0$, let $\mu^{\e}_{\xi_0}=\nu^{\e}\circ (\S_{\xi_0})^{-1}$. Then, the family $\{\mu^\e_{\xi_0}\}_{\e>0}$ of measures satisfies the LDP on the space $\mathrm{X}_2$ uniformly in ${\xi_0}\in\mathscr{T}$ with the rate function $\J_{\xi_0}$ defined as 
\begin{align*}
	\J_{\xi_0}(y)=\inf\{\I(x):\ x\in(\S_{\xi_0})^{-1}(y)\}, \ \ \text{ for all } \ \ y\in\mathrm{X_2}.
\end{align*}  
\end{proposition} 
Let us now move to the proof of ULDP for the system \eqref{1.2}, with respect to the bounded initial conditions $\xi_0$ in $\L^p(\1)$-norm, for $p>2$ with the help of uniform contraction principle (Proposition \ref{thrmU2}).

A solution to the following linear stochastic equation:
\begin{equation}\label{L2}
\left\{
\begin{aligned}
	\d \zeta(t,\x)+\A \zeta(t,\x)\d t&=\W(\d t,\d \x), \ \ (t,\x)\in(0,T]\times{\1},\\
	\zeta(0,\x)&=0,
\end{aligned}
\right.
\end{equation}is defined by  the variation of constant formula: 
\begin{align}\label{L3}
\zeta(t,\x)=\int_0^t\int_{\1}G(t-s,\x,\y)\W(\d s,\d\y), \ \ (t,\x)\in(0,T]\times{\1}.
\end{align}Furthermore, we have the following result:
\begin{lemma}[{\cite[Lemma 7]{BFMZ}}]For $a>0$ in \eqref{2.21} and $p>2$, we have 
\begin{align}\label{L4}
	\E\bigg[\sup_{t\in[0,T]}\|\zeta(t)\|_{\L^p}^p\bigg]<\infty.
\end{align}Moreover, $\zeta(\cdot,\cdot)$ has a continuous modification on $[0,T]\times{\1}$.
\end{lemma} Let us  set $\beta=\xi-\zeta$. Since the additive noise is independent of the unknown, we can see that $\beta$ satisfies the following equation:
\begin{align}\label{L5}
\frac{\partial}{\partial t}\beta(t)=\Delta\beta(t)-\nabla\cdot \q(\beta(t)+\zeta(t)), \ t\in[0,T],
\end{align}whose mild solution is given by
\begin{align}\label{L6}\nonumber
\beta(t,\x)&=\int_{\1}G(t,\x,\y)\xi_0(\y)\d \y\\&\quad +\int_0^t\int_{\1}\nabla_{\y}G(t-s,\x,\y)\cdot \q((\beta+\zeta)(s,\cdot))(\y)\d \y\d s,
\end{align}
where $\q$ is defined in \eqref{S1}. Let us recall the following result from \cite{BFMZ}.

\begin{lemma}[{\cite[Lemma 11]{BFMZ}}]\label{lemUF}
Let us assume that $\xi_0\in\L^p(\1)$, for $p>2$, $a>0$ in \eqref{2.21}. Then, for a constant $C_p>0$, we have
\begin{align*}
	\sup_{t\in[0,T]}\|\beta(t)\|_{\L^p}^p \leq \big(\|\xi_0\|_{\L^p}^p+C_1(\zeta)\big)e^{C_2(\zeta)},
\end{align*}where the constants $$C_1(\zeta)=C_pT\sup\limits_{t\in[0,T]}\|\zeta(t)\|_{\L^p}^{2p}\ \text{ and }\  \C_2(\zeta)=C_pT\bigg(1+\sup\limits_{t\in[0,T]}\|\zeta(t)\|_{\L^p}^2\bigg).$$
\end{lemma}

Next, we prove the local Lipschitz continuity of solutions of \eqref{L5} with respect to $\zeta\in\C([0,T];\L^p(\1)),$ for $p>2$.
\begin{lemma}\label{lem4.8}
For any $T>0$, $R_1>0$ and $ R_2>0$, there exists a positive constant $\wi{C}=\big(C_p(C_1+R_2^p)Te^{C_pT(1+C_1^{2/p}+R_2^2)} \big)^\frac{1}{p}$ such that the solution of \eqref{L5} satisfies 
\begin{align}\label{L7}
	\|\beta_{\zeta_1}-\beta_{\zeta_2}\|_{\C([0,T];\L^p(\1))} \leq \wi{C} \|\zeta_1-\zeta_2\|_{\C([0,T];\L^p(\1)},
\end{align}
for all $\xi_0\in\L^p(\1)$, for $p>2$ with $\|\xi_0\|_{\L^p}\leq R_1$ and $\zeta_1,\zeta_2\in\C([0,T];\L^p(\1))$ with \\ $\|\zeta_1\|_{\C([0,T];\L^p(\1))}\leq R_2$ and $\|\zeta_2\|_{\C([0,T];\L^p(\1))}\leq R_2$.
\end{lemma}

\begin{proof}From Lemma \ref{lemUF}, for every $R_1>0$ and $R_2>0$, there exists $C_1=C_1(R_1,R_2,T)=(R_1^p+C_pTR_2^{2p})e^{C_pT(1+R_2^2)}>0$, such that for all $\xi_0\in\L^p(\1)$ with $\|\xi_0\|_{\L^p}\leq R_1$ and $\zeta\in\C([0,T];\L^p(\1))$ with $\|\zeta\|_{\C([0,T];\L^p(\1))}\leq R_2$, 
\begin{align}\label{L07}
	\sup_{t\in[0,T]}\|\beta_\zeta(t)\|_{\L^p}^p\leq C_1, \ \mathbb{P}\text{-a.s.}
\end{align}

Let $\zeta_1$ and $\zeta_2$ be any two functions in the space $\C([0,T];\L^p(\1))$ such that $\|\zeta_1\|_{\C([0,T];\L^p(\1))}\\ \leq R_2$ and $\|\zeta_2\|_{\C([0,T];\L^p(\1))}\leq R_2$. We denote the corresponding solutions of \eqref{L5} by $\beta_{\zeta_i}$, for $i=1,2$. Then, $\beta_\zeta=\beta_{\zeta_1}-\beta_{\zeta_2}$, where $\zeta=\zeta_1-\zeta_2$, satisfies the following system for $t\in[0,T]$: 
\begin{equation}\label{L9}
	\left\{
	\begin{aligned}
		\frac{\partial}{\partial t}\beta_\zeta(t)&=\Delta\beta_\zeta(t)-\big[\nabla\cdot \q(\beta_{\zeta_1}(t)+\zeta_1(t))-\nabla\cdot \q(\beta_{\zeta_2}(t)+\zeta_2(t))\big], \\
		\beta_\zeta(0)&=0.
	\end{aligned}
	\right.
\end{equation}
Taking inner product with $|\beta_\zeta(\cdot)|^{p-2}\beta_\zeta(\cdot)$ to the first equation of the above system, we find
\begin{align}\label{L11}\nonumber
	&\frac{1}{p}\frac{\d }{\d t}\|\beta_\zeta(t)\|_{\L^p}^p +(p-1)\||\beta_\zeta(t)|^{\frac{p-2}{2}}\nabla\beta_\zeta(t)\|_{\L^2}^2\\&\quad =(p-1)\big(|\beta_\zeta(t)|^{p-2}\nabla \beta_\zeta(t), \q(\beta_{\zeta_1}(t)+\zeta_1(t))- \q(\beta_{\zeta_2}(t)+\zeta_2(t))\big),
\end{align}
for a.e. $t\in[0,T]$. 
Let us take   \begin{align}\label{L12}
	\b(\beta_{\zeta_1}+\zeta_1,\beta_{\zeta},\beta_{\zeta_1}+\zeta_1):=\big(|\beta_{\zeta}|^{p-2}\k*(\beta_{\zeta_1}+\zeta_1)\cdot\nabla\beta_{\zeta},\beta_{\zeta_1}+\zeta_1\big).
\end{align}
Then,  by fixing $|\beta_{\zeta}|^{p-2}$, we have 
\begin{align}\label{L13}\nonumber
	&	\b(\beta_{\zeta_1}+\zeta_1,\beta_{\zeta},\beta_{\zeta_1}+\zeta_1)-\b(\beta_{\zeta_2}+\zeta_2,\beta_{\zeta},\beta_{\zeta_2}+\zeta_2)\\&\nonumber=\b(\beta_{\zeta_1}+\zeta_1,\beta_{\zeta},\beta_{\zeta_1}+\zeta_1-(\beta_{\zeta_2}+\zeta_2))\\&\quad\nonumber+b((\beta_{\zeta_1}+\zeta_1)-(\beta_{\zeta_2}+\zeta_2),\beta_{\zeta},\beta_{\zeta_2}+\zeta_2)
	\\&\nonumber= 
	\b(\beta_{\zeta_1}+\zeta_1,\beta_{\zeta},\beta_{\zeta})+\b(\beta_{\zeta_1}+\zeta_1,\beta_{\zeta},\zeta)+\b(\beta_{\zeta}+\zeta,\beta_{\zeta},\beta_{\zeta_2}+\zeta_2)\\&\nonumber=
	\b(\beta_{\zeta_1}+\zeta_1,\beta_{\zeta},\zeta)+\b(\beta_{\zeta},\beta_{\zeta},\beta_{\zeta_2}+\zeta_2)+\b(\zeta,\beta_{\zeta},\beta_{\zeta_2}+\zeta_2)\\& =
	\big(|\beta_{\zeta}|^{p-2}\k*(\beta_{\zeta_1}+\zeta_1)\cdot\nabla\beta_\zeta,\zeta\big)+\big(|\beta_{\zeta}|^{p-2}\k*\beta_\zeta\cdot\nabla \beta_\zeta,\beta_{\zeta_2}+\zeta_2\big)\nonumber\\&\quad+\big(|\beta_{\zeta}|^{p-2}\k*\zeta\cdot\nabla\beta_\zeta,\beta_{\zeta_2}+\zeta_2\big),
\end{align}where we use the fact that $\b(\u,\beta_{\zeta},\beta_{\zeta})=\int_{\T^2}(\u\cdot\nabla\beta_{\zeta})|\beta_{\zeta}|^{p-2}\beta_{\zeta}\d x=0$, provided $\nabla\cdot\u=0$ (cf. \cite[Lemma 2.2]{HBBF}). Using \eqref{L12} and \eqref{L13} in \eqref{L11}, we obtain for a.e. $t\in[0,T]$
\begin{align}\label{L140}\nonumber
	&\frac{\d }{\d t}\|\beta_\zeta(t)\|_{\L^p}^p +p(p-1)\||\beta_\zeta(t)|^{\frac{p-2}{2}}\nabla\beta_\zeta(t)\|_{\L^2}^2\\&\nonumber =p(p-1)\big\{\big(\k*(\beta_{\zeta_1}(t)+\zeta_1(t))\cdot\nabla\beta_\zeta(t)|\beta_\zeta(t)|^{p-2},\zeta(t)\big)\\&\nonumber\quad+\big(\k*\beta_\zeta(t)\cdot\nabla \beta_\zeta|\beta_\zeta(t)|^{p-2},\beta_{\zeta_2}(t) +\zeta_2(t)\big)\\&\nonumber\quad+\big(\k*\zeta(t)\cdot\nabla\beta_\zeta(t)(\zeta)|\beta_\zeta(t)|^{p-2},\beta_{\zeta_2(t)}+\zeta_2(t)\big)\big\}\\&= I_1+I_2+I_3.
\end{align}Let us consider the term $I_1$ in the right hand side of the above equality, and estimate it using \eqref{2.16},  H\"older's and Young's inequalities as
\begin{align}\label{L141}\nonumber
	|I_1|& \leq C_p\||\beta_\zeta|^{\frac{p-2}{2}}\nabla\beta_\zeta\|_{\L^2}\||\beta_\zeta|^{\frac{p-2}{2}}\zeta\|_{\L^2}\|\k*(\beta_{\zeta_1}+\zeta_1)\|_{\mathbb{L}_\sigma^{\infty}}\\&\nonumber\leq \frac{1}{4}\||\beta_\zeta|^{\frac{p-2}{2}}\nabla\beta_\zeta\|_{\L^2}^2+C_p\|\beta_\zeta\|_{\L^p}^{p-2}\|\zeta\|_{\L^p}^2\|\beta_{\zeta_1}+\zeta_1\|_{\L^p}^2
	\\&\leq \frac{1}{4}\||\beta_\zeta|^{\frac{p-2}{2}}\nabla\beta_\zeta\|_{\L^2}^2+C_p\|\beta_\zeta\|_{\L^p}^p+C_p\|\beta_{\zeta_1}+\zeta_1\|_{\L^p}^{p}\|\zeta\|_{\L^p}^p,
\end{align}for $p>2$. Similarly, we can estimate the terms $I_2$ and $I_3$ from \eqref{L140} as 
\begin{align}\label{L142}
	|I_2| &\leq \frac{1}{4}\||\beta_\zeta|^{\frac{p-2}{2}}\nabla\beta_\zeta\|_{\L^2}^2+C_p\|\beta_{\zeta_2}+\zeta_2\|_{\L^p}^{2}\|\beta_\zeta\|_{\L^p}^p,\\ \label{L143}
	|I_3| &\leq \frac{1}{4}\||\beta_\zeta|^{\frac{p-2}{2}}\nabla\beta_\zeta\|_{\L^2}^2+C_p\|\beta_\zeta\|_{\L^p}^p+C_p\|\beta_{\zeta_2}+\zeta_2\|_{\L^p}^{p}\|\zeta\|_{\L^p}^p, 
\end{align}for $p>2$.

Combining \eqref{L12}-\eqref{L143}, we obtain 
\begin{align}\label{L14}\nonumber
	&\frac{\d }{\d t}\|\beta_\zeta(t)\|_{\L^p}^p +\frac{p(p-1)}{4}\||\beta_\zeta(t)|^{\frac{p-2}{2}}\nabla\beta_\zeta(t)\|_{\L^2}^2\\&\nonumber \leq C_p\big(1+\|\beta_{\zeta_2}(t)+\zeta_2(t)\|_{\L^p}^2\big)\|\beta_\zeta(t)\|_{\L^p}^p\\&\quad+C_p\big(\|\beta_{\zeta_1}(t)+\zeta_1(t)\|_{\L^p}^{p}+\|\beta_{\zeta_2}(t)+\zeta_2(t)\|_{\L^p}^{p}\big)\|\zeta(t)\|_{\L^p}^{p}.
\end{align}Applying Gronwall's inequality in \eqref{L14} and using the fact that $\|\zeta_1\|_{\C([0,T];\L^p(\1))}\leq R_2$ and $\|\zeta_2\|_{\C([0,T];\L^p(\1))}\leq R_2$, along with \eqref{L07}, we obtain
\begin{align}\label{L15}\nonumber
	&	\|\beta_\zeta(t)\|_{\L^p}^p\\&\nonumber\leq \bigg(C_p\int_0^T\big(\|\beta_{\zeta_1}(t)+\zeta_1(t)\|_{\L^p}^{p}+\|\beta_{\zeta_2}(t)+\zeta_2(t)\|_{\L^p}^{p}\big)\|\zeta(t)\|_{\L^p}^{p}\d t\bigg)\\&\nonumber\qquad\times\exp\bigg\{C_p\int_0^T\big(1+\|\beta_{\zeta_2}(t)+\zeta_2(t)\|_{\L^p}^2\big)\d s\bigg\}
	\\&\nonumber\leq C_pT\sup_{t\in[0,T]}\big\{\big(\|\beta_{\zeta_1}(t)\|_{\L^p}^p+\|\zeta_1(t)\|_{\L^p}^{p}+\|\beta_{\zeta_2}(t)\|_{\L^p}^p+\|\zeta_2(t)\|_{\L^p}^{p}\big)\|\zeta(t)\|_{\L^p}^{p}\big\}\\&\nonumber\qquad\times
	\exp{\bigg\{C_pT\bigg(1+\sup\limits_{t\in[0,T]}\big\{\|\beta_{\zeta_2}(t)\|_{\L^p}^2+\|\zeta_2(t)\|_{\L^p}^2\big\}\bigg)\bigg\}}\\&\leq \big(C_p(C_1+R_2^p)Te^{C_pT(1+C_1^{2/p}+R_2^2)} \big)\sup_{t\in[0,T]}\|\zeta(t)\|_{\L^p}^{p},
\end{align}where the constant $C_1=(R_1^p+C_pTR_2^{2p})e^{C_pT(1+R_2^2)}$. 
\end{proof}
Let us now establish the ULDP for the laws of solutions to the system \eqref{1.2}. For any given $T>0$ and the initial data $\xi_0\in\L^p(\1)$, define a map $\S_{\xi_0}(\cdot)$ from $\C([0,T];\L^p(\1))$ to $\C([0,T];\L^p(\1))$ such that 
\begin{align}\label{L015}
\S_{\xi_0}(\zeta)= \beta_\zeta, \ \ \text{ for all } \ \ \zeta \in\C([0,T];\L^p(\1)), \text{ for } p>2,
\end{align}where $\beta_\zeta$ is the solution to the system \eqref{L5}. For given $\phi\in\C([0,T];\L^p(\1))$ with $p>2$, we define 
\begin{align}\label{L016}
\J_{\xi_0}(\phi):=\inf\big\{\I(\psi): \ \psi\in\C([0,T];\L^p(\1)), \ \psi+\S_{\xi_0}(\psi)=\phi, \ \phi(0)=\xi_0\big\},
\end{align}where $\I(\cdot)$ is the rate function defined in \eqref{U5}.

Let us move to the main result for this subsection, that is, ULDP for the system \eqref{1.2} in $\C([0,T];\L^p(\1))$, for $p>2$. 
\begin{theorem}\label{thrmUL}
For any given $T>0$ and the initial data $\xi_0\in\L^p(\1)$, for $p>2$. Let $\xi^\e$ be the solution of \eqref{1.2}, and $\mu^\e_{\xi_0}$ be the law of $\xi^\e$ in the space $\C([0,T];\L^p(\1))$, for $p>2$. Then, the family of measures $\{\mu^\e_{\xi_0}\}_{\e>0}$ satisfies an LDP in $\C([0,T];\L^p(\1))$ , for $p>2$ with the rate function $\J_{\xi_0}(\cdot)$ defined in \eqref{L016} uniformly with respect to the initial conditions $\xi_0$ belongs to a bounded subsets of $\L^p(\1)$, for $p>2$.
\end{theorem}
\begin{proof}
For the given initial data $\xi_0\in\L^p(\1)$, with $p>2$, let $\S_{\xi_0}(\cdot)$ and $\J_{\xi_0}(\cdot)$ be the mappings defined in \eqref{L015} and \eqref{L016}, respectively. From Lemma \ref{lem4.8}, we infer that $\zeta+\S_{\xi_0}(\zeta)$ is locally Lipschitz continuous for $\zeta\in\C([0,T];\L^p(\1))$, with $p>2$, uniformly over the initial conditions $\xi_0$ lying in a bounded subset in $\L^p(\1)$-norm.

Let $\zeta^\e$ be the solution to the system \eqref{U6} and $\nu^\e$ be the law  of $\zeta^\e$. Then, we have 
\begin{align*}\xi^\e=\beta_{\zeta^\e}^\e+\zeta^\e=(\I+\S_{\xi_0})(\zeta^\e).\end{align*}
Since, $\mu^\e_{\xi_0}$ is the law of $\xi^\e$, we have $\mu^\e_{\xi_0}=\nu^\e\circ(\I+\S_{\xi_0})^{-1}$. From \cite[Theorem 12.16]{DaZ}, we know that $\{\nu^\e\}_{\e>0}$ satisfies the LDP on the space $\C([0,T];\L^p(\1))$, for $p>2$, with the rate function $\I(\cdot)$ defined in \eqref{U5}. By the uniform contraction principle (Proposition \ref{thrmU2}), we obtain that the family of measures $\{\mu^\e_{\xi_0}\}_{\e>0}$ satisfies the LDP on the space $\C([0,T];\L^p(\1))$, for $p>2$ uniformly over the initial conditions $\xi_0$ lying in bounded subsets in $\L^p(\1)$-norm, with the rate function defined by 
\begin{align*}
	\J_{\xi_0}(\phi)&=\inf\big\{\I(\psi): \psi\in(\I+\S_{\xi_0})^{-1}(\phi)\big\}\\&
	=\inf\big\{\I(\psi): \psi\in\C([0,T];\L^p(\1)), \psi+\S_{\xi_0}(\psi)=\phi\big\},
\end{align*}and hence the proof completes.
\end{proof}
Now, our aim is to prove ULDP for the law of the solutions to the system \eqref{1.2} in $\C([0,T]\times\1)$. For that we only need to prove the following result in $\C([0,T]\times\1)$.
\begin{lemma}\label{lem3.9}
For any $T>0, \ R_1>0$ and $R_2>0$, the solution of \eqref{L5} satisfies:
\begin{align}\label{326}
	\|\beta_{\zeta_1}-\beta_{\zeta_2}\|_{\C([0,T]\times\1)} \leq C\left(T,R_1,R_2\right)\|\zeta_1-\zeta_2\|_{\C([0,T]\times\1)},
\end{align}for all $\xi\in\C(\1)$, with $\|\xi_0\|_{\C(\1)}\leq R_1$ and $\zeta_1,\zeta_2\in\C([0,T]\times\1)$ with \begin{align}\label{327}
	\|\zeta_1\|_{\C([0,T]\times\1)}\leq R_2\ \text{ and }\ \|\zeta_2\|_{\C([0,T]\times\1)}\leq R_2.
\end{align}
\end{lemma}
\begin{proof}
From \cite[Theorem 1]{BFMZ}, we know that $\xi\in\C([0,T]\times\1)$, whenever the initial data $\xi_0\in\C(\1)$ and \cite[Lemma 7]{BFMZ}  implies that $\zeta$  has a continuous modification on $[0,T]\times{\1}$, and hence $\beta\in \C([0,T]\times\1)$. Therefore, there exists a constant $\vi{C}>0$ such that
\begin{align}\label{L071}
	\|\beta_\zeta\|_{\C([0,T]\times\1)}\leq \vi{C} (R_1,R_2,T), \ \mathbb{P}\text{-a.s.}
\end{align}

Let $\zeta_1$ and $\zeta_2$ be any two functions in the space $\C([0,T]\times\1)$ such that \eqref{327} be satisfied. We denote the corresponding solutions of \eqref{L6} by $\beta_{\zeta_i}$, for $i=1,2$. Then, for $\beta_\zeta=\beta_{\zeta_1}-\beta_{\zeta_2}$, where $\zeta=\zeta_1-\zeta_2$, we have
\begin{align*}	&	|\beta_{\zeta_1}(t,x)-\beta_{\zeta_2}(t,x)|\\&\nonumber=\bigg|\int_0^t\int_{\1}\nabla_yG(t-s,x,y)\cdot\big(\q((\beta_{\zeta_1}+\zeta_1)(s,\cdot))(y)-\q((\beta_{\zeta_2}+\zeta_2)(s,\cdot))(y)\big)\d y\d s\bigg|
	\\&\nonumber\leq \bigg|\int_0^t\int_{\1}\nabla_yG(t-s,x,y)\cdot (\k*(\beta_{\zeta_1}+\zeta_1)(s,y))(\beta_{\zeta}+\zeta)(s,y)\d y\d s\bigg| 
	\\&\nonumber\quad +\bigg|\int_0^t\int_{\1}\nabla_yG(t-s,x,y)\cdot(\k*(\beta_{\zeta}+\zeta)(s,y))(\beta_{\zeta_2}+\zeta_2)\d y\d s\bigg| 
	\\&\nonumber 
	\leq C
	\int_0^t\int_{\1}\big|\nabla_yG(t-s,x,y)\big|\big| (\beta_{\zeta_1}+\zeta_1)(s,y)\big|\big|(\beta_{\zeta}+\zeta )(s,y)\big|\d y\d s \\&\nonumber\quad +C\int_0^t\int_{\1}\big|\nabla_yG(t-s,x,y)\big|\big|(\beta_{\zeta}+\zeta)(s,y)\big|\big|(\beta_{\zeta_2}+\zeta_2)(s,y)\big|\d y\d s
	\\&\nonumber 	\leq C	\int_0^t\int_{\1}\big|\nabla_yG(t-s,x,y)\big|\big\{\big| (\beta_{\zeta_1}+\zeta_1)(s,y)\big|+\big| (\beta_{\zeta_2}+\zeta_2)(s,y)\big|\big\}\big|\beta_{\zeta}(s,y)\big|\d y\d s \\&\quad +C\int_0^t\int_{\1}\big|\nabla_yG(t-s,x,y)\big|\big\{\big|(\beta_{\zeta_1}+\zeta_1)(s,y)\big|+(\beta_{\zeta_2}+\zeta_2)(s,y)\big|\big|\zeta(s,y)\big|\d y\d s.
\end{align*}
Applying Gronwall's inequality in above inequality, and using \eqref{2.9} and \eqref{L071}, we obtain 
\begin{align}\label{328}\nonumber 
	&\sup_{(t,x)\in[0,T]\times\1}	|\beta_{\zeta_1}(t,x)-\beta_{\zeta_2}(t,x)|\\& \nonumber \leq 
	C\left(T^\frac{1}{2}\big\{ \|\beta_{\zeta_1}+\zeta_1\|_{\C([0,T]\times\1)}+\|\beta_{\zeta_2}+\zeta_2\|_{\C([0,T]\times\1)}\big\}\|\zeta\|_{\C([0,T]\times\1)}\right)\\&\nonumber \qquad \times \exp\left\{T^{\frac{1}{2}}\big\{ \|\beta_{\zeta_1}+\zeta_1\|_{\C([0,T]\times\1)}+\|\beta_{\zeta_2}+\zeta_2\|_{\C([0,T]\times\1)}\big\}\right\}
	\\&	\leq C\left(T^\frac{1}{2}\big(\vi{C}+R_2\big)\|\zeta\|_{\C([0,T]\times\1)}\right) \times e^{2T^{\frac{1}{2}}(\vi{C}+R_2)},
\end{align}
and \eqref{326} follows. 
\end{proof}
Let us now establish the ULDP for the laws of solutions to the system \eqref{1.2}. For any given $T>0$ and the initial data $\xi_0\in\C(\1)$, define a map $\S_{\xi_0}(\cdot)$ from $\C([0,T]\times\1)$ to $\C([0,T]\times\1)$ such that 
\begin{align}\label{L0151}
\S_{\xi_0}(\zeta)= \beta_\zeta, \ \ \text{ for all } \ \ \zeta \in\C([0,T]\times\1),
\end{align}where $\beta_\zeta$ is the solution to the system \eqref{L6}. For given $\phi\in\C([0,T]\times\1)$, we define 
\begin{align}\label{L0161}
\J_{\xi_0}(\phi):=\inf\big\{\I(\psi): \ \psi\in\C([0,T]\times\1), \ \psi+\S_{\xi_0}(\psi)=\phi, \ \phi(0)=\xi_0\big\},
\end{align}where $\I(\cdot)$ is the rate function defined in \eqref{U5}.
\begin{theorem}\label{thrmUL1}
For any given $T>0$ and the initial data $\xi_0\in\C(\1)$. Let $\xi^\e$ be the solution of \eqref{1.2}, and $\mu^\e_{\xi_0}$ be the law of $\xi^\e$ in the space $\C([0,T]\times\1)$. Then, the family of measures $\{\mu^\e_{\xi_0}\}_{\e>0}$ satisfies an LDP in $\C([0,T]\times\1)$ with the rate function $\J_{\xi_0}(\cdot)$ defined in \eqref{L0161} uniformly with respect to the initial conditions $\xi_0$ belongs to bounded subsets  of $\C(\1)$.
\end{theorem}
\begin{proof}
Since the proof of this theorem is based on Theorem	\ref{thrmUL} with \eqref{L0151} and \eqref{L0161} as substitutes for \eqref{L015} and \eqref{L016}, respectively, we are not supplying any additional proof.
\end{proof}

\section{Finite dimensional noise}\label{FD}\setcounter{equation}{0}
In this section, we consider the system \eqref{FD1} and establish the well-posedness results. If $\xi^\e$ is a solution of the integral equation \eqref{FD2} in the sense of Walsh (\cite{JBW}), then by using \cite[Proposition 3.7]{IGCR}, one can prove that $\xi^\e$ has a continuous modification:
\begin{align}\label{FD02}\nonumber
\int_{\1}\xi^\e(t,\y)\phi(\y)\d \y&=\int_{\1}\xi_0(\y)\phi(\y)\d \y+
\int_0^t\int_{\1} \xi^\e(t,\y)\Delta\phi(\y)\d \y\d s
\\&\nonumber\quad-\int_0^t\int_{\1}\q(\xi^\e(s,\cdot))(\y)\cdot\nabla\phi(\y)\d \y\d s\\&\quad+\sqrt{\e}\sum_{j=1}^{n}\int_0^t\int_{\1}\sigma_j(s,\y,\xi^\e(s,\y))\phi(\y)\d \y\d\W^j(s),   \ \ \P\text{-a.s.}, 
\end{align} for every test function $\phi\in\mathrm{H}^b$ with $b>2$, and all time $t\in[0,T]$.

Let us introduce the assumptions on the noise coefficient $\sigma_j,$ for $j=1,\ldots,n$:
\begin{hypothesis}\label{hyp1}
The function $\sigma_j:[0,T]\times\mathbb{R}^2\times\R\to \R$ is a measurable function, satisfying the following conditions:
\begin{equation}\label{FD3}
	\left\{\begin{aligned}
		|\sigma_j(t,\x,r)|&\leq K(1+|r|), \\ |\sigma_j(t,\x,r)-\sigma_j(t,\x,s)|&\leq L|r-s|,
	\end{aligned}
	\right.
\end{equation}for $j=1,\ldots,n$, for all $t\in[0,T]$, $\x\in{\1}$ and $r,s\in\R$, where $K,L$ are some positive constants.
\end{hypothesis}Under the above assumption, we state our main result of this section.
\begin{theorem}\label{thrmex}
Let  $\xi_0\in\L^p({\1})$, for $p>2$ be given and Hypothesis \ref{hyp1} hold. Then, there exists a unique $\L^p$-valued $\mathscr{F}_t$-adapted continuous process $\xi^\e$ of the integral equation \eqref{FD2}. Moreover, if the initial data $\xi_0\in\C({\1})$, then the solution $\xi^\e$ admits a modification which is  a space-time continuous process.
\end{theorem}
\subsection{Existence and uniqueness results} In this subsection, we prove Theorem \ref{thrmex} by establishing several auxiliary results.  
In order to establish a solution to the system \eqref{FD1}, we use a truncation method. Let $\Pi_R(\cdot)$ be a $\C^1(\R)$ function such that 
\begin{equation}\label{FD4}
\Pi_R(r)=	\left\{
\begin{aligned}
	&	1, \ \ \text{  if  } \ \ |r|\leq R,\\
	&0, \ \ \text{  if  }\ \  |r|\geq  R+1,
\end{aligned}
\right.
\end{equation} and $|\Pi_R'(r)|\leq 2$ for all $r\in\R$. Also, by the mean value theorem we have $|\Pi_R(r_1)-\Pi_R(r_2)|\leq 2|r_1-r_2|$, for all $r_1,r_2\in\R$.  Let us introduce the truncated system for $	\xi_R^\e(t,\x)=\nabla^{\perp}\cdot\u_R^\e(t,\x), \ (t,\x)\in[0,T]\times {\1}$ as 
\begin{equation}\label{FD5}
\left\{
\begin{aligned}
	\frac{\partial \xi_R^\e}{\partial t}(t,\x)&+\u_R^\e(t,\x)\cdot\nabla\xi_R^\e(t,\x)\Pi_R(\|\xi_R^\e(t)\|_{\L^p})-\Delta\xi_R^\e(t,\x)\\&=\sqrt{\e}\sum_{j=1}^{n}\sigma_{j,R}(t,\x,\xi_R^\e(t,x))\frac{\d }{\d t}\W^j(t), \ (t,\x)\in[0,T]\times{\1},\\
	\nabla\cdot \u_R^\e(t,\x)&=0, \ (t,\x)\in[0,T]\times{\1},\\
	\xi_R^\e(0,\x)&=\xi_0(\x), \ \x\in{\1}.
\end{aligned}
\right.
\end{equation}
We  write the mild formulation of the truncated system \eqref{FD5} in sense of Walsh (see \cite{JBW}) as
\begin{align}\label{FD6}
\nonumber
\xi_R^\e(t,\x)&=\int_{\1} G(t,\x,\y)\xi_0(\y)\d\y+\int_0^t\int_{\1}\nabla_{\y} G(t-s,\x,\y) \cdot \q_R(\xi_R^\e(s,\cdot))(\y)\d\y\d s\\&\quad +\sqrt{\e}\sum_{j=1}^{n}\int_0^t\int_{\1} G(t-s,\x,\y)\sigma_{j,R}(t,\y,\xi_R^\e(t,\y))\d\y\d\W^j(t), \ \ \P\text{-a.s.,}
\end{align}for all $t\in[0,T]$, where $\q_R(\xi^\e(\cdot))(\cdot)=\Pi_R(\|\xi^\e(\cdot)\|_{\L^p})\q(\xi^\e(\cdot))(\cdot)$ and $ \sigma_{j,R}(\cdot,\cdot,\xi^\e(\cdot,\cdot))=\Pi_R(\|\xi^\e(\cdot)\|_{\L^p})\sigma_{j}(\cdot,\cdot,\xi^\e(\cdot,\cdot)).$

Set \begin{align}\label{FD7}	\nonumber
\mathcal{A}_1\xi_R^\e(t,\x)&:=\int_{\1}G(t,\x,\y)\xi_0(\y)\d\y,\\	\nonumber
\mathcal{A}_2\xi_R^\e(t,\x)&:=\int_0^t\int_{\1}\nabla_{\y} G(t-s,\x,\y) \cdot \q_R(\xi_R^\e(s,\cdot))(\y)\d\y\d s,\\ 	\nonumber
\mathcal{A}_3\xi_R^\e(t,\x)&:=\sum_{j=1}^{n}\int_0^t\int_{\1} G(t-s,\x,\y)\sigma_{j,R}(s,y,\xi_R^\e(s,y))\d\y\d\W^j(s),\\
\mathcal{A}\xi_R^\e(t,\x)&:=	\mathcal{A}_1\xi_R^\e(t,\x)+	\mathcal{A}_2\xi_R^\e(t,\x)+\sqrt{\e}	\mathcal{A}_3\xi_R^\e(t,\x).
\end{align}Let us define a Banach space $\mathscr{H}$ consisting of the adapted processes $u:\Omega\times[0,T]\to\L^p(\1)$ such that 
\begin{align*}
\|u\|_{\mathscr{H}}^p:=\sup_{t\in[0,T]}e^{-\lambda pt}\E\big[\|u(t)\|_{\L^p}^p\big],
\end{align*}where $\lambda>0$  is defined in \eqref{417} below. 

For $\xi^\e\in\L^p(\1)$ for $p>2$, we obtain that the function $\q_R$ is defined from $\L^p(\1)$ to $\LL_{\sigma}^p(\1)$ for any $p>2$. Let us recall a result from \cite{BFMZ}.
\begin{lemma}[{\cite[Lemma 9]{BFMZ}}]\label{lem5.2}
Fix $R\geq 1$ and $p>2$. Then, there exist positive constants $C_p$ and $C_{R,p}$ such that the following hold:
\begin{align}\label{FD8}
	\|\q_R(\xi)\|_{\LL_{\sigma}^p}&\leq C_p(R+1)^2, \\ 
	\label{FD10}
	\|\q_R(\xi)-\q_R(\tilde{\xi})\|_{\LL_{\sigma}^p}&\leq C_{R,p}\|\xi-\tilde{\xi}\|_{\L^p}, 
\end{align}for all $\xi, \tilde{\xi}\in\L^p(\1)$.
\end{lemma}

\begin{proposition}\label{prop5.3}
Let us assume that $\xi_0\in\L^p(\1)$, for $p>2$, and $R\geq 1$. Then, there exists a unique $\L^p$-valued $\mathscr{F}_t$-adapted continuous process $\xi_R^\e$ satisfying the truncated integral equation \eqref{FD6} such that 
\begin{align}\label{FD11}
	\E\bigg[\sup_{t\in[0,T]}\|\xi_R^\e(t)\|_{\L^p}^p\bigg]\leq C(R,T).
\end{align}
\end{proposition}
\begin{proof}
The proof of this proposition is based on a fixed point argument. The proof is divided into two parts. In the first part, we establish that the operator $\mathcal{A}$ is well-defined in the Banach space $\mathscr{H}$. In the second part, we show that $\mathcal{A}$ is a contraction map and then the existence and uniqueness of truncated integral equation \eqref{FD6} follows from fixed point arguments.

\vspace{2mm}
\noindent
\textbf{Claim 1.} {\it The operator $\mathcal{A}$ is well-defined on the Banach space $\mathscr{H}$}.

From \eqref{FD7}, we have 
\begin{align}\label{FD12}
	\|	\mathcal{A}\xi_R^\e(t)\|_{\L^p}^p\leq C_p\left(	\|\mathcal{A}_1\xi_R^\e(t)\|_{\L^p}^p+	\|\mathcal{A}_2\xi_R^\e(t)\|_{\L^p}^p+	\|\mathcal{A}_3\xi_R^\e(t)\|_{\L^p}^p\right).
\end{align}Using \eqref{2.10} and Young's inequality, we obtain  
\begin{align}\label{FD13}\nonumber
	\|\mathcal{A}_1\xi_R^\e\|_{\mathscr{H}} &=\sup_{t\in[0,T]}e^{-\lambda t}\|\mathcal{A}_1\xi_R^\e(t)\|_{\L^p} \\&\leq \sup_{t\in[0,T]}e^{-\lambda t} \big\{\|G(t,0,\cdot)\|_{\L^1}\|\xi_0\|_{\L^p}\big\}<\infty.
\end{align}Applying \eqref{R2} for $\beta=1$ and $\alpha=p$, we obtain 
\begin{align*}
	\|\mathcal{A}_2\xi_R^\e(t)\|_{\L^p} \leq \int_0^t(t-s)^{-\frac{1}{2}}\|\q_R(\xi_R^\e(s,\cdot))\|_{\LL_{\sigma}^p}\d s \leq C_p(R+1)^2t^\frac{1}{2},
\end{align*}which implies 
\begin{align}\label{FD14}
	\|\mathcal{A}_2\xi_R^\e\|_{\mathscr{H}}^p=\sup_{t\in[0,T]}e^{-\lambda t}\|\mathcal{A}_2\xi_R^\e(t)\|_{\L^p}^p <\infty.
\end{align}Now, we consider the term $\E\big[\|\mathcal{A}_3\xi_R^\e(t)\|_{\L^p}^p\big]$, and estimate it using Hypothesis \ref{hyp1}, Fubini's theorem, Burkholder-Davis-Gundy's (BDG)   (see \cite[Theorem 1.1.7]{SVLBLR}), \cite[Proposition 3.5]{IGCR}, H\"older's, Minkowski's and Young's inequalities  as 
\begin{align}\label{FD15}\nonumber	&\E\bigg[\|\mathcal{A}_3\xi_R^\e(t)\|_{\L^p}^p\bigg]\\&\nonumber=\e^\frac{p}{2}\E\bigg[
	\int_{\1}\bigg|\sum_{j=1}^{n}\int_0^t\int_{\1}G(t-s,\x,\y)\sigma_{j,R}(s,\x,\xi_R^\e(t,x))\d\y\d\W^j(s) \bigg|^p\d \x\bigg]\\&\nonumber\leq C  \e^\frac{p}{2}\E\bigg[
	\int_{\1}\bigg|\sum_{j=1}^{n}\int_0^t\bigg|\int_{\1}G(t-s,\x,\y)\sigma_{j,R}(s,\y,\xi_R^\e(s,\y))\d\y\bigg|^2\d s \bigg|^\frac{p}{2}\d \x\bigg]\\& \nonumber\leq C\e^\frac{p}{2}\E\bigg[\sum_{j=1}^{n}\int_0^t\bigg|\int_{\1}\bigg|\int_{\1}G(t-s,\x,\y)\sigma_{j,R}(s,,\y,\xi_R^\e(s,\y))\d \y\bigg|^p\d \x\bigg|^\frac{2}{p}\d s\bigg]^{\frac{p}{2}}
	\\&\nonumber \leq C\e^\frac{p}{2}\E\bigg[\sum_{j=1}^{n}\int_0^t\big\|G(t-s,\cdot,\cdot)*\sigma_{j,R}(s,\cdot,\xi_R^\e(s,\cdot))\big\|_{\L^p}^2\d s\bigg]^{\frac{p}{2}}
	\\&\nonumber \leq C\e^\frac{p}{2}\E\bigg[\sum_{j=1}^{n}\int_0^t\big\|G(t-s,\cdot,\cdot)\|_{\L^1}^2\|\sigma_{j,R}(s,,\cdot,\xi_R^\e(s,\cdot))\big\|_{\L^p}^2\d s\bigg]^{\frac{p}{2}}
	\\&\nonumber \leq C\e^\frac{p}{2}\E\bigg[\sum_{j=1}^{n}\int_0^t\|\sigma_{j,R}(s,,\cdot,\xi_R^\e(s,\cdot))\big\|_{\L^p}^2\d s\bigg]^{\frac{p}{2}}\\&\nonumber\leq 
	C\e^\frac{p}{2}(tn)^{\frac{p}{2}-1}\int_0^t\E\left[\big(1+\Pi_R^p(\|\xi^\e_R(s)\|_{\L^p})\|\xi_R^\e(s)\|_{\L^p}^p)\right]\d s\\&\leq C\e^\frac{p}{2}K^{\frac{p}{2}-1}T^{\frac{p}{2}}\big(1+(R+1)^p\big).
\end{align}Combining \eqref{FD12}-\eqref{FD15}, the claim follows.

\vspace{2mm}
\noindent
\textbf{Claim 2.} {\it $\mathcal{A}$ is a contraction}. Let $\xi_{R,1}^\e,\xi_{R,2}^\e\in\mathscr{H}$. Without loss of generality, we can assume that $\|\xi_{R,1}^\e\|_{\L^p}\leq \|\xi_{R,2}^\e\|_{\L^p}$. To end this, we need to establish that $\mathcal{A}$ is a contraction on the Banach space $\mathscr{H}$. Let us start with the term $\|\mathcal{A}_2\xi_{R,1}^\e(\cdot)-\mathcal{A}_2\xi_{R,2}^\e(\cdot)\|_{\L^p}$ and  estimate it  using \eqref{R2} with $\alpha=p,\  \beta=1$ to obtain 
\begin{align*}
	&	\|\mathcal{A}_2\xi_{R,1}^\e(t)-\mathcal{A}_2\xi_{R,2}^\e(t)\|_{\L^p} \\&\nonumber= \bigg\|\int_0^t\int_{\1} \nabla_{\y} G(t-s,\x,\y) \cdot \big(\q_R(\xi_{R,1}^\e(s,\cdot)-\q_R(\xi_{R,2}^\e(s,\cdot))\big)(\y)\d\y\d s\bigg\|_{\L^p}\\&\nonumber \leq C\int_0^t(t-s)^{-\frac{1}{2}}\|\q_R(\xi_{R,1}^\e(s,\cdot)-\q_R(\xi_{R,2}^\e(s,\cdot))\|_{\LL_{\sigma}^p}\d s,
\end{align*}for all $t\in[0,T]$. From the above inequality and Lemma \ref{lem5.2} (see \eqref{FD10}), we obtain for all $t\in[0,T]$
\begin{align}\label{FD16}\nonumber &	\|\mathcal{A}_2\xi_{R,1}^\e-\mathcal{A}_2\xi_{R,2}^\e\|_{\mathscr{H}}^p\\&\nonumber \leq C_{R,p}\sup_{t\in[0,T]}e^{-\lambda pt}\E\bigg[\bigg(\int_0^t(t-s)^{-\frac{1}{2}}\|\xi_{R,1}^\e(s)-\xi_{R,2}^\e(s)\|_{\L^p}\d s\bigg)^p\bigg] \\&\nonumber \leq
	C_{R,p}T^{\frac{p-1}{2p}}\sup_{t\in[0,T]}\bigg(\int_0^t(t-s)^{-\frac{1}{2}}e^{-\lambda(t-s)}e^{-\lambda s}\E\big[\|\xi_{R,1}^\e(s)-\xi_{R,2}^\e(s)\|_{\L^p}^p \big]\d s\bigg)\\& \leq  C_{R,p}T^{\frac{p-1}{2p}}\left(\frac{\Gamma(\frac{1}{2})}{\lambda^{\frac{1}{2}}}\right)\|\xi_{R,1}^\e-\xi_{R,2}^\e\|_{\mathscr{H}}^p= C_{R,p}T^{\frac{p-1}{2p}}\left(\frac{\sqrt{\pi}}{\lambda^{\frac{1}{2}}}\right)\|\xi_{R,1}^\e-\xi_{R,2}^\e\|_{\mathscr{H}}^p,
\end{align}for $p>2$, where $\Gamma(\cdot)$ denotes the Gamma function. 

Using similar arguments as in \eqref{FD15}, Hypothesis \ref{hyp1} (Lipschitz continuity of $\sigma_j$) and Minkowski's inequality, we find 
\begin{align*}
	&\E\big[	\|\mathcal{A}_3\xi_{R,1}^\e(t)-\mathcal{A}_3\xi_{R,2}^\e(t)\|_{\L^p} ^p\big]
	\\&\nonumber \leq C\e^\frac{p}{2}\E\bigg[\sum_{j=1}^{n}\int_0^t\big\|G(t-s,\cdot,\cdot)*(\sigma_{j,R}(s,\cdot,\xi_{R,1}^\e(s,\cdot))-\sigma_{j,R}(s,\cdot,\xi_{R,2}^\e(s,\cdot)))\big\|_{\L^p}^2\d s\bigg]^{\frac{p}{2}}
	\\&\nonumber \leq C\e^\frac{p}{2}\E\bigg[\sum_{j=1}^{n}\int_0^t\big\|G(t-s,\cdot,\cdot)\|_{\L^1}^2\|\sigma_{j,R}(s,\cdot,\xi_{R,1}^\e(s,\cdot))-\sigma_{j,R}(s,\cdot,\xi_{R,2}^\e(s,\cdot)\big\|_{\L^p}^2\d s\bigg]^{\frac{p}{2}}\\&\nonumber\leq C\e^\frac{p}{2}\E\bigg[\sum_{j=1}^{n}\int_0^t\|\sigma_{j,R}(s,\cdot,\xi_{R,1}^\e(s,\cdot))-\sigma_{j,R}(s,\cdot,\xi_{R,2}^\e(s,\cdot)\big\|_{\L^p}^2\d s\bigg]^{\frac{p}{2}}\\&\leq C_{\e,n,p,L}\E\bigg[\int_0^t\|\xi_{R,1}^\e(s)-\xi_{R,2}^\e(s)\|_{\L^p}^2\d s\bigg]^\frac{p}{2} \\&\leq C_{\e,n,p,L}\left(\bigg\{\E\bigg[\int_0^t\|\xi_{R,1}^\e(s)-\xi_{R,2}^\e(s)\|_{\L^p}^2\d s\bigg]^\frac{p}{2}\bigg\}^\frac{2}{p}\right)^\frac{p}{2} \\&\leq 
	C_{\e,n,p,L}\left( \int_0^t\bigg\{\E\big[\|\xi_{R,1}^\e(s)-\xi_{R,2}^\e(s)\|_{\L^p}^p\big]\bigg\}^\frac{2}{p}\d s\right)^\frac{p}{2}
	\\&\leq 	C_{\e,n,p,L} \bigg(\frac{1}{2\lambda}(e^{2\lambda t}-1)\bigg)^\frac{p}{2}\sup_{s\in[0,t]} \bigg\{e^{-\lambda ps}\E\big[\|\xi_{R,1}^\e(s)-\xi_{R,2}^\e(s)\|_{\L^p}^p\big]\bigg\},
\end{align*}for $p\geq2$. From the above inequality, we obtain 
\begin{align}\label{FD17}
	\|\mathcal{A}_3\xi_{R,1}^\e-\mathcal{A}_3\xi_{R,2}^\e\|_{\mathscr{H}} ^p\leq C_{\e,n,p,L}\bigg(\frac{1}{2\lambda}(1-e^{-2\lambda t})\bigg)^\frac{p}{2}\|\xi_{R,1}^\e-\xi_{R,2}^\e\|_{\mathscr{H}}^p.
\end{align} Combining \eqref{FD16}-\eqref{FD17}, and choosing $\lambda>0$ sufficiently large so that 
\begin{align}\label{417}
	C_{R,p}T^{\frac{p-1}{2p}}\left(\frac{\sqrt{\pi}}{\lambda^{\frac{1}{2}}}\right)+C_{\e,n,p,L}\bigg(\frac{1}{2\lambda}(1-e^{-2\lambda T})\bigg)^\frac{p}{2} <1,
\end{align}for $p>2$. For the above value of $\lambda$, the mapping $\mathcal{A}$ is a contraction on the Banach space $\mathscr{H}$. 

Therefore an application of the Banach fixed point theorem yields  the existence of  a unique fixed point for the operator $\mathcal{A}$ and this gives the existence of a unique solution of the truncated integral equation and the estimate \eqref{FD11} holds from Claim 1. 
\end{proof}
By Proposition \ref{prop5.3}, we obtain the existence of unique local solution of the integral equation \eqref{FD2}, that is, up to a stopping time $	\tau_{R}:=\inf\{t\geq0:\|\xi^\e(t)\|_{\L^p}\geq R\}\wedge T$.  To prove that the solution we achieved is a global one, we must demonstrate that a uniform bound holds.  Consider the following integral equation:
\begin{align}\label{FD017}\nonumber
\xi_m^\e(t,\x)&=\int_{\1} G(t,\x,\y)\xi_{0,m}(\y)\d\y+\int_0^t\int_{\1}\nabla_{\y} G(t-s,\x,\y) \cdot \q_m(\xi_m^\e(s,\cdot))(\y)\d\y\d s\\&\nonumber\quad +\sqrt{\e}\sum_{j=1}^{n}\int_0^t\int_{\1} G(t-s,\x,\y)\sigma_{j,m}(s,y,\xi_m^\e(s,y))\d\y\d\W^j(s)\\&=:I_{1,m}+I_{2,m}+I_{3,m},
\end{align}where we have taken the sequence of bounded Borel measurable functions $\q_m(t,\x,r)=\q(t,\x,r)$ and $\sigma_{j,m}(t,\x,r)=\sigma_j(t,\x,r)$, for $|r|\leq m$, and $\q_m(t,\x,r)=\sigma_{j,m}(t,\x,r)=0$, for $|r|\geq m+1$ such that they are globally Lipschitz in $r\in\R$ with constants independents of $m$.

The following result provides a uniform energy estimate for $\xi_m^\e$.
\begin{lemma}\label{lem5.4}For each $m\in\N$, $p>2$, there exists a constant $C>0$ such that 
\begin{align}\label{FD18}\nonumber
	\E\bigg[\sup_{t\in[0,T]}\|\xi_m^\e(t)\|_{\L^p}^p &+2p(p-1)\int_0^T\||\xi_m^\e(s)|^{\frac{p-2}{2}}\nabla\xi_m^\e(s)\|_{\L^2}^2\d s\bigg]\\& \leq \big(2\|\xi_0\|_{\L^p}^p+C_{\e,p,T}\big) e^{C_{\e,p,T}}.
\end{align}
\end{lemma}
\begin{proof}
Applying It\^o's formula to the function $|\cdot|^p$ for the process $\xi_m^\e(\cdot,\x)$ for $\x\in{\1}$ and then taking integration over the spatial domain ${\1}$ (cf. \cite{IG}), we deduce for all $t\in[0,T]$
\begin{align}\label{FD19}\nonumber
	\|\xi_m^\e(t)\|_{\L^p}^p &= \|\xi_0\|_{\L^p}^p-p(p-1)\int_0^t\||\xi_m^\e(s)|^{\frac{p-2}{2}}\nabla\xi_m^\e(s)\|_{\L^2}^2\d s\\&\nonumber \quad +p(p-1)\int_0^t \big(|\xi_m^\e(s,y)|^{p-2}\nabla\xi_m^\e(s,\y), \q_m(\xi_m^\e(s,\y))\big)\d s\\&\nonumber \quad+p\sqrt{\e}\sum_{j=1}^{n}\int_0^t\int_{\1} |\xi_m^\e(s,y)|^{p-2} \xi_m^\e(s,\y)\sigma_{j,m}(s,\y,\xi_m^\e(s,\y))\d \y\d\W^j(s)\\&\nonumber\quad +\frac{\e}{2}p(p-1)\sum_{j=1}^{n}\int_0^t\int_{\1} |\xi_m^\e(s,\y)|^{p-2}\sigma_{j,m}^2(s,\y,\xi_m^\e(s,\y))\d \y\d s\\&=: \|\xi_0\|_{\L^p}^p-p(p-1)I_1+p(p-1)I_2+I_3+I_4.
\end{align}Consider the term $I_2$ and using $\q(\xi_m^\e)=\xi^\e_m(\k*\xi_m^\e)$, $\u_m^\e=\k*\xi_m^\e$ and \cite[Lemma 2.2]{HBBF}  as 
\begin{align}\label{FD019}
	I_2&= \int_0^t\big(|\xi_m^\e(s,y)|^{p-2}\xi_m^\e(s,y)\nabla \xi_m^\e(s,y),\u_m^\e(s,y)\big)\d s=0,
\end{align}since $\u_m^\e$ is a divergence free velocity field.

Now, we consider the term $\E\bigg[\sup\limits_{t\in[0,T]}|I_3|\bigg]$ and estimate it using BDG (see \cite[Theorem 1.1.7]{SVLBLR}), H\"older's and Young's inequalities as
\begin{align}\label{FD20}\nonumber
	&	\E\bigg[\sup_{t\in[0,T]}|I_3|\bigg] \\&\nonumber=p\sqrt{\e} \E\bigg[\sup_{s\in[0,T]}\bigg|\sum_{j=1}^{n}\int_0^t\int_{\1} |\xi_m^\e(s,y)|^{p-2} \xi_m^\e(s,\y)\\&\qquad\times\nonumber\sigma_{j,m}(s,\y,\xi_m^\e(s,\y))\d \y\d\W^j(s)\bigg|\bigg]\\& \nonumber\leq C_{\e,p}\E\bigg[\int_0^T\bigg(\int_{\1}|\xi_m^\e(s,y)|^{p-2} \xi_m^\e(s,\y)\sigma_{j,m}(s,\y,\xi_m^\e(s,\y))\d \y \bigg)^2\d s\bigg]^\frac{1}{2}\\&\nonumber\leq C_{\e,p}\E\bigg[\int_0^T\bigg(\int_{\1} \big(1+|\xi_m^\e(s,\y)|^p\big)\d \y\bigg)^2\d s\bigg]^\frac{1}{2}\\&\nonumber\leq C_{\e,p}\E\bigg[T+\int_0^T\|\xi_m^\e(s)\|_{\L^p}^{2p}\d s\bigg]^\frac{1}{2}\\&\leq 
	\frac{1}{2}\E\bigg[\sup_{t\in[0,T]}\|\xi_m^\e(t)\|_{\L^p}^p\bigg]+C_{\e,p}\int_0^T\E\big[\|\xi_m^\e(s)\|_{\L^p}^p\big]\d s+C_{\e,p,T}.
\end{align}
Using Hypothesis \ref{hyp1}, we estimate the term $I_4$ as 
\begin{align}\label{FD21}
	I_4\leq C_{\e,p}\int_0^t\int_{\1} \big(1+|\xi_m^\e(s,\y)|^p\big)\d \y\d s\leq C_{\e,p}\int_0^t\|\xi_m^\e(s)\|_{\L^p}^p\d s+C_{\e,p,T}.
\end{align}Combining \eqref{FD19}-\eqref{FD21}, we find 
\begin{align}\label{FD22}\nonumber
	\E\bigg[\sup_{t\in[0,T]}\|\xi_m^\e(t)\|_{\L^p}^p\bigg] &+2p(p-1)\E\bigg[\int_0^T\||\xi_m^\e(s)|^{\frac{p-2}{2}}\nabla\xi_m^\e(s)\|_{\L^2}^2\d s\bigg]\\& \leq 2\|\xi_0\|_{\L^p}^p+ C_{\e,p}\int_0^t\E\big[\|\xi_m^\e(s)\|_{\L^p}^p\big]\d s+C_{\e,p,T}. 
\end{align}Applying Gronwall's inequality in \eqref{FD22}, we obtain the required uniform energy estimate \eqref{FD18}.
\end{proof}Let us go back to the integral equation \eqref{FD2}, and establish the existence and uniqueness result stated in Theorem \ref{thrmex}. 
\begin{proof}[Proof of Theorem \ref{thrmex}]
\textsl{Uniqueness.} We employ a stopping time argument to prove this. Let us assume that $\xi_1^\e$ and $\xi_2^\e$ be any two solutions of the integral equation \eqref{FD2}. Then, $\xi_1^\e$ and $\xi_2^\e$ are the solutions of the truncated integral equation \eqref{FD6}. Let us define the following stopping times
\begin{align*}
	\tau_{R}^1:=\inf\{t\geq0:\|\xi_1^\e(t)\|_{\L^p}\geq R\}\wedge T \ \ \text{ and } \ \ 	\tau_{R}^2:=\inf\{t\geq0:\|\xi_2^\e(t)\|_{\L^p}\geq R\}\wedge T.
\end{align*}By Proposition \ref{prop5.3}, we obtain $\xi_1^\e(t)=\xi_2^\e(t), \ t\in[0,\tau_R^1\wedge\tau_R^2]$. By the uniform energy estimate \eqref{FD18} for $\xi_i^\e$, $i=1,2$, and an application of Markov's inequality yields (see Remark \ref{rem4.6} below)
\begin{align*}
	\P\big(\tau_R^1\wedge\tau_R^2<T\big)\to0, \ \ \text{ as }\ \ R\to\infty
\end{align*} and the uniqueness follows. 

\vspace{2mm}
\noindent
\textsl{Existence.} Let us now show the existence of the solution $\xi^\e$ to the integral equation \eqref{FD2} in the interval $[0,T]$. The proof of this part is based on Skorokhod's representation theorem (see \cite[Theorem 3]{RMD}) and  \cite[Lemma 1.1]{IGNK}. By Lemma \ref{lem25}, we obtain that the sequence of  processes 
\begin{align}\label{FD24}
	I_{2,m}&:=\int_0^t\int_{\1}\nabla_{\y} G(t-s,\x,\y) \cdot \q_m(\xi_m^\e(s,\cdot))(\y)\d\y\d s, \ \ t\in[0,T]
\end{align}is tight in $\C([0,T];\L^p({\1}))$, for $p>2$, uniformly in $m$. By \cite[Corollary 3.6]{IGCR}, the sequence of processes
\begin{align}\label{FD25}
	I_{3,m}&:=\sum_{j=1}^{n}\int_0^t\int_{\1} G(t-s,\x,\y)\sigma_{j,m}(s,y,\xi_m^\e(s,y))\d\y\d\W^j(s), \ \ t\in[0,T]
\end{align}is also tight in $\C([0,T];\L^p({\1}))$, for $p>2$. Note that the sequence 
\begin{align}\label{FD26}
	I_{1,m}&:=\int_{\1}G(t,\x,\y)\xi_{0,m}(\y)\d\y, \ \ t\in[0,T]
\end{align}is tight in $\C([0,T];\L^p({\1}))$. Therefore, one can conclude that $\xi_m^\e(t)=\sum\limits_{i=1}^{3}I_{i,m}, \  t\in[0,T]$ is tight in $\C([0,T];\L^p({\1}))$, for $p>2$, uniformly in $m$. 

By Skorokhod's representation theorem, for a given pair of subsequences $\{\xi_{m}^\e\}$ and $\{\xi_l^\e\}$, there exist subsequences $m_i$ and $l_i$ (denoted by $\xi_{m_i}^\e$, and $\xi_{l_i}^\e$) and a sequence of random variables $\z_i^\e:=(\vi{\xi}_{i}^\e,\bar{\xi}_{i}^\e,\bar{\W}^j_i)$, $i=1,2,\ldots,$ in $\mathcal{I}=\C([0,T];\L^p({\1}))\times\C([0,T];\L^p({\1}))\times\C([0,T]\times{\1})$, for $p>2$ in some probability space\\ $(\bar{\Omega},\bar{\mathscr{F}},\{\bar{\mathscr{F}}_t\}_{t\geq0},\bar{\P})$ such that the sequence $\z_i$ converges to $\z^\e=(\vi{\xi}^\e,\bar{\xi}^\e,\bar{\W}^j),\ \bar{\P}$-a.s.,  in $\mathcal{I}$ as $i\to\infty$ and the laws of $z^{\e}_i$ and $(\xi_{m_i}^{\e},\xi_{l_i}^{\e},\W^j)$ coincide. Here, the two random fields $\bar{\W}^j$ and $\bar{\W}^j_i$ are Wiener process defined on different stochastic basis $(\bar{\Omega},\bar{\mathscr{F}},\{\bar{\mathscr{F}}_t\}_{t\geq0},\bar{\P})$ and $(\bar{\Omega},\bar{\mathscr{F}},\{\bar{\mathscr{F}}_t^i\}_{t\geq0},\bar{\P})$, respectively, where $\bar{\mathscr{F}}_t$ and $\bar{\mathscr{F}}_t^i$ are the completion of the $\sigma$-fields generated by $\z^\e(t,\x)$ and $\z_i^\e(t,\x)$ for all $(t,\x)\in[0,T]\times{\1}$, respectively. Using \cite[Lemma 1.1]{IGNK}, our aim is to show that $\z_i^\e$ converges to $\z^\e$ weakly. From \cite[Proposition 3.7]{IGCR}, we have the equivalency between weak and mild solutions. Letting $i\to\infty$, in the weak formulation \eqref{FD02} with $\vi{\xi}_i^\e$ in place of $\xi^\e$, for every smooth function $\phi\in\mathrm{H}^b$ with $b>2$, we have
\begin{align}\label{FD27}\nonumber
	\int_{\1}\vi{\xi}^\e(t,\y)\phi(\y)\d \y&=\int_{\1}\xi_0(\y)\phi(\y)\d \y+
	\int_0^t\int_{\1} \vi{\xi}^\e(t,\y)\Delta\phi(\y)\d \y\d s
	\\&\nonumber\quad-\int_0^t\int_{\1}\q(\vi{\xi}^\e(s,\cdot))(\y)\cdot\nabla\phi(\y)\d \y\d s\\&\quad+\sqrt{\e}\sum_{j=1}^{n}\int_0^t\int_{\1}\sigma_j(s,\y,\vi{\xi}^\e(s,\y))\phi(\y)\d \y\d\W^j(s),   \ \ \P\text{-a.s.}, 
\end{align}  for all $t\in[0,T]$ on the probability space $(\bar{\Omega},\bar{\mathscr{F}},\{\bar{\mathscr{F}}_t^i\}_{t\geq0},\bar{\P})$. Similarly, \eqref{FD27} hols for $\bar{\xi}^\e$ also. Then by the uniqueness of the solution we deduce that $\vi{\xi}^\e=\bar{\xi}^\e$. Using \cite[Lemma 1.1]{IGNK}, we conclude that $\xi^\e_m$ converges in $\C([0,T];\L^p({\1}))$, for $p>2$, in probability to some random element $\xi^\e\in\C([0,T];\L^p({\1}))$, for $p>2$. 

Let us now assume that the initial data $\xi_0\in\C({\1})$. Then our solution $\xi^\e$ of the integral equation \eqref{FD2} is the sum of three terms. The first term $\int_{\1}G(t,\x,\y)\xi_0(\y)\d \y$ is continuous, using the properties of heat kernel $G$ (see Lemma \ref{thrm2.1}(ii)). We know that  $\xi_0\in \L^p(\1)$, for any $p\in[1,\infty]$, since $\xi_0\in\C({\1})$. Therefore, $\xi\in\C([0,T];\L^p(\T^2))$, for all $p>2$ and \eqref{2.17} implies $\q(\xi)\in\C([0,T];\LL_\sigma^p({\1}))$. Fixing $p>4$, we conclude   by Lemma \ref{lem25} (2)  that $J\q(\xi)\in\C([0,T]\times{\1})$. For $p>2$, the tightness of final term $\sum_{j=1}^{n}\int_0^t\int_{\1} G(t-s,\x,\y)\sigma_{j,m}(s,y,\xi_m^\e(s,y))\d\y\d\W^j(s)$ in the space $\C([0,T]\times{\1})$ is immediate from \cite[Corollary 3.6]{IGCR}. Hence, the solution $\xi^\e$ has a stochastic modification which is continuous in $(t,\x)\in[0,T]\times{\1}$. This completes the proof of Theorem \ref{thrmex}. 
\end{proof}

\begin{remark}\label{rem4.6}
One can show the existence of solution in the following way also: We first consider the solution $\xi_R^\e$ of the truncated integral equation \eqref{FD6} for any $R\in\N$. Let us define a sequence of stopping times 
\begin{align*}\tau_R:=\inf\{t\geq0:\|\xi_R^\e(t)\|_{\L^p}\geq R\}\wedge T, \ \ \text{ for } \ \ R\in\N.
\end{align*}From Proposition \ref{prop5.3}, we know that $\xi_R^\e(t,\cdot)$ is the unique solution of the integral equation \eqref{FD2} for any $t\in[0,\tau_R]$. By the uniqueness of solution, we infer that  for any $P\geq R$, $\tau_R\leq \tau_P$ and $\xi_P^\e(t)=\xi_R^\e(t)$ for $t\leq \tau_R$. Let us define $\xi^\e(t)=\xi_R^\e(t)$ for $t\leq \tau_R$. Using the similar arguments we are able to construct a solution $\xi^\e(t)$ of the integral equation \eqref{FD2} in the random interval $[0,\tau_\infty]$, where $\tau_\infty=\sup\limits_{R\geq 0}\tau_R$.
Now, our aim is to establish that $\tau_\infty=T,\ \P$-a.s., which is equivalent to show that $\lim\limits_{R\to\infty}\P\big(\tau_R<T\big)=0$.

From the uniform energy estimate \eqref{FD18}, we have 
\begin{align*}
	\E\bigg[\sup_{t\in[0,T]}\|\xi_R^\e(t)\|_{\L^p}^p\bigg] \leq C\big(\|\xi_0\|_{\L^p}^p+1\big),
\end{align*}for $p>2$. Since $\tau_R\leq \tau_\infty$, we have 
\begin{align*}
	\P\big(\tau_\infty<T\big)\leq \P\big(\tau_R<T\big) =\P\bigg(\sup_{t\in[0,T]}\|\xi_R^\e(t)\|_{\L^p}^p> R^p\bigg). 
\end{align*}Applying Markov's inequality in the above inequality, we find 
\begin{align*}
	\P\big(\tau_R<T\big) \leq \frac{1}{R^p}\E\bigg[\sup_{t\in[0,T]}\|\xi_R^\e(t)\|_{\L^p}^p\bigg]\leq \frac{C\big(\|\xi_0\|_{\L^p}^p+1\big)}{R^p}.
\end{align*}Passing $R\to\infty$ leads to our required result, that is, $\tau_\infty=T,\ \P$-a.s. 
\end{remark}

\section{Uniform large Deviation Principle: Finite-dimensional Noise}\label{UDLP}\setcounter{equation}{0}
In this section, we first recall the definition of equicontinuous uniform Laplace principle (EULP)  (see \cite{MS}, for ULDP, see Subsection \ref{SULDP}). Later, we state a sufficient  condition under which  EULP holds. We start with the following additional information:

One should note that FWULDP depends on the choice of the Polish space $(\EE,d)$.	Recall that  a family $\B\subset \C_b(\EE)$ (space of bounded and continuous functions from $\EE$ to $\R$) of functions from $\EE\to\R$ is equibounded and equicontinuous if 
\begin{align*}
\sup_{g\in\B}\sup_{\psi\in\EE}|g(\psi)| <\infty \ \ \text{ and }\ \ \lim_{\delta\to0}\sup_{g\in\B}\sup_{\substack{\psi_1,\psi_2\in\EE,\\d(\psi_1,\psi_2)<\delta}}|g(\psi_1)-g(\psi_2)|=0,
\end{align*}
respectively. 
\begin{definition}[EULP]\label{def6.2}
Let $\mathscr{T}$ be the collection of subsets of  $\EE_0$ and $b(\e)$ be a function which converges to $0$ as $\e\to0$.  	A family of $\EE$-valued random variables $\{\mathrm{Y}_{\xi_0}^\e\}_{\e>0}$ is said to satisfy an \emph{equicontinuous uniform Laplace principle} with the speed $b(\e)$ and the rate function $\I_{\xi_0}(\cdot)$, uniformly over $\mathscr{T}$ if for any $\mathrm{E}\in\mathscr{T}$ and any equibounded and equicontinuous family $\B\subset\C_b(\EE)$, 
\begin{align}\label{6.1}
	\lim_{\e\to0}\sup_{g\in\B}\sup_{\xi_0\in\mathrm{E}}\bigg|b(\e)\log\E \bigg[\exp\bigg(-\frac{g(\mathrm{Y}_{\xi_0}^\e)}{b(\e)}\bigg)\bigg]+\inf_{h\in\EE}\big\{g(h)+\I_{\xi_0}(\psi)\big\}\bigg|=0.
\end{align}
\end{definition}
An equivalency between   ULDP and  EULP is established in \cite{MS}. 
\begin{proposition}[{\cite[Theorem 2.10]{MS}}]\label{thrm6.3}
EULP and FWULDP are equivalent.
\end{proposition}
Let us provide a sufficient condition for a family of random variables to satisfy  EULP. Then by Proposition \ref{thrm6.3}, we obtain FWULDP. Let $\PP_2$ be the set of $\mathscr{F}_t$-adapted controls $v:[0,T]\to\R^n$ such that $\P\big(\|v\|_{\L^2(0,T;\R^n)}<\infty\big)=1$. For $M>0$, let $\PP_2^M\subset\PP_2$ be the set of admissible controls such that 
\begin{align*}
\PP_2^M:=\big\{v\in\PP_2:\P\big(\|v\|_{\L^2(0,T;\R^n)}\leq M\big)=1\big\}.
\end{align*}

\begin{condition}\label{cond}
Assume that there exists a family of measurable maps\\ $\mathscr{G}^\e:\C([0,T];\R^n)\to\EE$, indexed by $\e\in(0,1]$, $\xi_0\in\EE_0$. Let $\W(\cdot)$ be a $n$-dimensional Wiener process and $\mathrm{Y}_v^\e:=\mathscr{G}^\e\big(\xi_0,\sqrt{\e}\W+\int_0^{\cdot}v(s)\d s\big)$. Define $\mathscr{G}^0:\C([0,T];\R^n)\to\EE$, the limiting case of $\mathscr{G}^\e$ as $\e\to0$. Let $\mathscr{T}$ be collection  of the subset of $\EE_0$ such that for any $\mathrm{E}\in \mathscr{T}, \ M>0$ and $\delta>0$, \small{
	\begin{align}\label{6.2}
		\lim_{\e\to0}\sup_{\xi_0\in\mathrm{E}}\sup_{v\in\PP_2^M}\P\left\{d\bigg(\mathscr{G}^\e\bigg(\xi_0,\sqrt{\e}\W+\int_0^{\cdot}v(s\d s\bigg),\mathscr{G}^0\bigg(\xi_0,\int_0^{\cdot}v(s)\d s\bigg)\bigg)>\delta\right\}=0.
\end{align}}
\end{condition}
\begin{theorem}[{\cite[Theorem 2.13]{MS}}]\label{thrm6.4}
Assume that Condition \ref{cond} holds. Then, the family $\mathrm{Y}_{\xi_0}^\e:=\mathscr{G}^\e(\xi_0,\sqrt{\e}\W)$ satisfies the EULP uniformly over $\mathscr{T}$ with the speed $b(\e)=\e$, and the rate function $\I_{\xi_0}:\EE\to\R$,
\begin{align}\label{6.3}
	\I_{\xi_0}(h)=\inf\bigg\{\frac{1}{2}\int_0^T|v(s)|_{\R^n}^2\d s: \ h:=\mathscr{G}^0\bigg(\xi_0,\int_0^{\cdot}v(s)\d s\bigg)\bigg\}.
\end{align}
\end{theorem}If we verify Condition \ref{cond}, then by Theorem \ref{thrm6.4}, the family $\{\mathrm{Y}_{\xi_0}^\e\}_{\e>0}$ satisfies the EULP and by equivalency we obtain that the family $\{\mathrm{Y}_{\xi_0}^\e\}_{\e>0}$ satisfies ULDP. Thus, our main aim is to verify Condition \ref{cond}.
\subsection{Controlled and skeleton equation} The unique mild solution to the system \eqref{FD1} is a probabilistically strong solution, since \eqref{S3} is satisfied and  an estimate similar to \eqref{FD18} can be established  for $\xi^{\e}$. Let us fix $\EE=\C([0,T];\L^p({\1}))$, $\EE_0=\L^p({\1})$ and the solution map of the integral equation \eqref{FD2} be represented by $\xi^\e=\mathscr{G}^\e(\xi_0,\sqrt{\e}\W)$.

Let us denote the solution $\xi_v^\e(t,\x):=\mathscr{G}^\e\big(\xi_0,\sqrt{\e}\W(t)+\int_0^{t}v(s)\d s\big)$ of the following stochastic controlled integral (SCI) equation
\begin{align}\label{6.5}\nonumber
\xi_v^\e(t,\x)&=\int_{\1} G(t,\x,\y)\xi_0(\y)\d\y+\int_0^t\int_{\1}\nabla_{\y} G(t-s,\x,\y) \cdot \q(\xi_v^\e(s,\cdot))(\y)\d\y\d s\\&\nonumber\quad +\sqrt{\e}\sum_{j=1}^{n}\int_0^t\int_{\1} G(t-s,\x,\y)\sigma_j(s,\y,\xi_v^\e(s,\y))\d \y\d\W^j(s)\\&\quad +\sum_{j=1}^{n}\int_0^t\int_{\1} G(t-s,\x,\y)\sigma_j(s,\y,\xi_v^\e(s,\y))v_j(s)\d \y\d s,\  \P\text{-a.s.},
\end{align}for all $t\in[0,T]$. It we choose $\e=0$, then the above SCI equation changes to skeleton equation (limiting case). Let us denote the solution $\xi_v^0(t,\x):=\mathscr{G}^0\big(\xi_0,\int_0^{t}v(s)\d s\big)$ of the following skeleton equation: 
\begin{align}\label{6.6}\nonumber
\xi_v^0(t,\x)&=\int_{\1} G(t,\x,\y)\xi_0(\y)\d\y+\int_0^t\int_{\1}\nabla_{\y} G(t-s,\x,\y) \cdot \q(\xi_v^0(s,\cdot))(\y)\d\y\d s\\&\quad +\sum_{j=1}^{n}\int_0^t\int_{\1} G(t-s,\x,\y)\sigma_j(s,\y,\xi_v^0(s,\y))v_j(s)\d \y\d s, \  \text{ for all } \ t\in[0,T].
\end{align}For $\psi\in\C([0,T];\L^p({\1}))$, we define the rate function
\begin{align}\label{6.7}
\I_{\xi_0}(\psi):=\inf_{\big\{v\in\L^2(0,T;\R^n):\psi=\mathscr{G}^0\big(\xi_0,\int_0^{\cdot}v(s)\d s\big)\big\}}\bigg\{\frac{1}{2}\int_0^T|v(s)|^2_{\R^n}\d s\bigg\},
\end{align}where $\psi$ satisfies the skeleton equation \eqref{6.6}.

The following results discusses the existence and uniqueness of solution to the above SCI and skeleton equations.
\begin{theorem}\label{thrm6.5}
Assume that Hypothesis \ref{hyp1} holds and $v\in\PP_2^M$. Then, for any initial data $\xi_0\in\L^p({\1})$, for $p>2$, there exists a unique solution $\xi_v^\e\in\C([0,T];\L^p({\1}))$ of the SCI equation \eqref{6.5}.
\end{theorem}
\begin{proof}
The main ingredient of this proof is the Girsanov's theorem (see \cite[Theorem 10.14]{DaZ}).

For any given $v\in\PP_2^M$, we define 
\begin{align}\label{6.9}
	\frac{\d \vi{\P}}{\d\P}:=\exp\bigg\{-\frac{1}{\sqrt{\e}}\int_0^Tv(s)\d\W(s)-\frac{1}{2\e}\int_0^T|v(s)|_{\R^n}^2\d s\bigg\},
\end{align}where $\W(\cdot)$ is a finite-dimensional Wiener process. Then, the stochastic process 
\begin{align*}
	\vi{\W}(t):=\W(t)+\frac{1}{\sqrt{\e}}\int_0^tv(s)\d s, \ \ t\in[0,T],
\end{align*}is a finite-dimensional Wiener process under the probability measure $\vi{\P}$. Also, 
\begin{align*}
	\exp\bigg\{-\frac{1}{\sqrt{\e}}\int_0^Tv(s)\d\W(s)-\frac{1}{2\e}\int_0^T|v(s)|^2_{\R^n}\d s\bigg\},
\end{align*}is an exponential martingale and $\vi{\P}$ is another probability measure on the probability space $(\Omega,\mathscr{F},\{\mathscr{F}_t\}_{t\geq 0},\P)$ and $\vi{\P}$ is equivalent to the measure $\P$. Using Girsanov's theorem, we obtain hat $\vi{\W}(\cdot)$ is a real valued finite-dimensional Wiener process with respect to the measure $\P$. By Theorem \ref{thrmex}, we obtain the existence and uniqueness of the solution to the SCI equation \eqref{6.5} with respect to the new probability measure $\vi{\P}$. Hence, we conclude the well-posedness to the SCI equation \eqref{6.5} with respect to the probability measure $\P$ also.
\end{proof}
\begin{theorem}\label{thrm6.6}
Assume that Hypothesis \ref{hyp1} holds and $v\in\L^2(0,T;\R^n)$. Then, for any initial data $\xi_0\in\L^p({\1})$, for $p>2$, there exists a unique solution $\xi_v^0\in\C([0,T];\L^p({\1}))$, to the skeleton equation \eqref{6.6}.
\end{theorem}
\begin{proof}
Since $v\in\L^2(0,T;\R^n)$, and 
\begin{align*}
	&	\bigg|\sum_{j=1}^{n}\int_0^t\int_{\1} G(t-s,\x,\y)\sigma_j(s,\y,\xi_v^0(s,\y))v_j(s)\d y\d s\bigg| \\&\leq C \sum_{j=1}^{n}\int_0^t \int_{\1}\big(1+|\xi_v^0(s,\y)|\big)|v_j(s)|\d \y\d s \\&\leq C\sum_{j=1}^{n} \bigg(\int_0^t|v_j(s)|^2\d s\bigg)^\frac{1}{2}\bigg\{t^\frac{1}{2}+\bigg(\int_0^t\bigg(\int_{\1}|\xi_v^0(s,\y)|\d\y\bigg)^2\d s\bigg)^\frac{1}{2}\bigg\}
	\\&\leq C\|v\|_{\L^2(0,T;\R^n)}\bigg\{T^\frac{1}{2}+T^{\frac{1}{4}}\bigg(\int_0^T\|\xi_v^0(s)\|_{\L^2}^2\d s\bigg)^\frac{1}{4}\bigg\}<\infty,
\end{align*}the well-posedness of the skeleton equation \eqref{6.6} can be shown by using similar arguments as in the proof of Theorem \ref{thrmex}, with some minor modification in the proof.
\end{proof}

\begin{remark}Let us assume that $\xi_0\in\C(\1)$. Using the well-posedness established in Section \ref{FD}, and \cite[Lemma 3.1, Corollary 3.2 and Proposition 3.5]{IGCR}, we conclude that the solution $\xi_v^\e$ of the integral equation \eqref{6.5}  has a modification which is a space-times continuous process. Moreover,  the solution $\xi_v^0$ of the integral equation  \eqref{6.6} has a version which belongs to the space $\C([0,T]\times\1)$.
\end{remark}
\subsection{Main Theorems} In this subsection, we establish that the family of solutions $\xi^\e$ of the integral equation \eqref{FD02} satisfies ULDP in two different topologies $\C([0,T];\L^p({\1}))$, for $p>2$ and $\C([0,T]\times{\1})$.
\begin{theorem}[ULDP in {$\C([0,T];\L^p({\1}))$}]\label{thrm6.7}
Let $p>2$. The family $\{\xi_v^\e\}_{\e>0}$ satisfies the ULDP on $\C([0,T];\L^p({\1}))$, with rate function $\I_{\xi_0}(\cdot)$ given by \eqref{6.7}, where the uniformity is over $\L^p({\1})$-bounded sets of initial conditions, if the following hold:
\begin{enumerate}
	\item For any $N>0$, $t_0\geq 0$, and $\delta>0$
	\begin{align*}
		\liminf_{\e\to0}\inf_{\|\xi_0\|_{\L^p}\leq N}\inf_{\psi\in\Lambda_{\xi_0}(t_0)}\bigg\{\e\log\P\big(\|\xi_v^\e-\psi\|_{\C([0,T];\L^p({\1}))}<\delta\big)+\I_{\xi_0}(\psi)\bigg\}\geq 0.
	\end{align*}
	\item For any $N>0$, $t_0\geq 0$, and $\delta>0$
	\begin{align*}
		\limsup_{\e\to0}\sup_{\|\xi_0\|_{\L^p}\leq N}\sup_{s\in[0,t_0]}\bigg\{\e\log\P\big(\mathrm{dist}_{\C([0,T];\L^p({\1}))}(\xi_v^\e,\Lambda_{\xi_0(s)})>\delta\big)+s\bigg\}\leq 0.
	\end{align*}
	
\end{enumerate}
\end{theorem}
\begin{theorem}[ULDP in {$\C([0,T]\times{\1})$}]\label{thrm6.8}
The family $\{\xi_v^\e\}_{\e>0}$ satisfies the ULDP on $\C([0,T]\times{\1})$, with rate function $\I_{\xi_0}(\cdot)$ given by \eqref{6.7}, where the uniformity is over $\C({\1})$-bounded sets of initial conditions, if the following hold:
\begin{enumerate}
	\item For any $N>0$, $t_0\geq 0$, and $\delta>0$
	\begin{align*}
		\liminf_{\e\to0}\inf_{\|\xi_0\|_{\L^p}\leq N}\inf_{\psi\in\Lambda_{\xi_0}(t_0)}\bigg\{\e\log\P\big(\|\xi_v^\e-\psi\|_{\C([0,T]\times{\1})}<\delta\big)+\I_{\xi_0}(\psi)\bigg\}\geq 0.
	\end{align*}
	\item For any $N>0$, $t_0\geq 0$, and $\delta>0$
	\begin{align*}
		\limsup_{\e\to0}\sup_{\|\xi_0\|_{\L^p}\leq N}\sup_{s\in[0,t_0]}\bigg\{\e\log\P\big(\mathrm{dist}_{\C([0,T]\times{\1})}(\xi_v^\e,\Lambda_{\xi_0(s)})>\delta\big)+s\bigg\}\leq 0.
	\end{align*}		
\end{enumerate}
\end{theorem}
\begin{proposition}[Global estimate]\label{prop6.9}
For any $p>2$, there exists a constant $C>0$ depending on $p,M$, and $T$, such that for $\e\in(0,1], \ \xi_0\in \L^p({\1})$ and $v\in\PP_2^M$, 
\begin{align}\label{6.10}\nonumber
	\E\bigg[\sup_{t\in[0,T]}\|\xi_v^\e(t)\|_{\L^p}^p\bigg] +&2p(p-1)\E\bigg[\int_0^T\||\xi_v^\e(s)|^{\frac{p-2}{2}}\nabla\xi_v^\e(s)\|_{\L^2}^2\d s\bigg]\\&\leq C \big(1+\|\xi_0\|_{\L^p}^p\big).
\end{align}
\end{proposition}         
\begin{proof}The proof of this proposition can be obtained in a similar manner as in the proof of \eqref{FD20} in Lemma \ref{lem5.4}.
\end{proof} 
\begin{corollary}\label{cor6.10}
The random field $\sup\limits_{t\in[0,T]}\|\xi_v^\e(t)\|_{\L^p}^p,\ p>2$ is bounded in probability, that is, for any given $T>0,\ M>0$, and $N>0$, 
\begin{align}\label{6.11}
	\lim_{R\to\infty} \sup_{\|\xi_0\|_{\L^p}\leq N}\sup_{v\in\PP_2^M}\sup_{\e\in(0,1]}\P \bigg(\sup_{t\in[0,T]}\|\xi_v^\e(t)\|_{\L^p}^p>R\bigg)=0.		
\end{align}
\end{corollary}
\begin{proof}
	Applying Markov's inequality along with the global estimate \eqref{6.10} yields the required result \eqref{6.11}.
\end{proof}

\begin{corollary}[Uniform convergence in probability]\label{cor6.11}
For any $T>0,\ \delta>0, \ N>0, \ M>0$ and $p>2$
\begin{align}\label{6.12}
	\lim_{\e\to0}\sup_{\|\xi_0\|_{\L^p}\leq N} \sup_{v\in\PP_2^M}\P\bigg(\|\xi_v^\e-\xi_v^0\|_{\C([0,T];\L^p({\1}))}>\delta\bigg)=0.
\end{align}
\end{corollary}
\begin{proof}
For any $v\in\PP_2^M$, $\|\xi_0\|_{\L^p}\leq N$, for $p>2$ and $\e\in(0,1]$, we have 
\begin{align}\label{6.13}\nonumber
	\xi_v^\e(t,\x)&-\xi_v^0(t,\x)\\&\nonumber= \int_0^t\int_{\1}\nabla_{\y} G(t-s,\x,\y) \cdot\big( \q(\xi_v^\e(s,\cdot))- \q(\xi_v^0(s,\cdot))(\y)\big)\d\y\d s\\&\nonumber\quad +\sum_{j=1}^{n}\int_0^t\int_{\1} G(t-s,\x,\y)\big(\sigma_j(s,\y,\xi_v^\e(s,\y))-\sigma_j(s,\y,\xi_v^0(s,\y))\big)v_j(s)\d \y\d s\\&\nonumber\quad +\sqrt{\e}\sum_{j=1}^{n}\int_0^t\int_{\1} G(t-s,\x,\y)\sigma_j(s,\y,\xi_v^\e(s,\y))\d \y\d\W^j(s)\\& =: I_{1,\e}+I_{2,\e}+\sqrt{\e}I_{3,\e}.
\end{align}We consider $I_{1,\e}$ and estimate it using  \eqref{2.16} and \eqref{R2} with $\alpha=p,\  \beta=1$, to find 
\begin{align}\label{6.14}\nonumber
	&	\|I_{1,\e}(t)\|_{\L^p}\\&\nonumber = \bigg\|\int_0^t\int_{\1}\nabla_{\y} G(t-s,\x,\y) \cdot\big( \q(\xi_v^\e(s,\cdot))- \q(\xi_v^0(s,\cdot))(\y)\big)\d\y\d s \bigg\|_{\L^p}\\&\nonumber\leq C \int_0^t(t-s)^{-\frac{1}{2}}\|\q(\xi_v^\e(s,\cdot))-\q(\xi_v^0(s,\cdot))\|_{\LL_{\sigma}^p}\d s
	\\&\nonumber\leq C \int_0^t(t-s)^{-\frac{1}{2}}\|(\k*\xi_v^\e(s))(\xi_v^\e(s)-\xi_v^0(s))+(\k*\xi_v^{\e}(s)\\&\qquad\nonumber-\k*\xi_v^0(s))\xi_v^0(s)\|_{\LL_{\sigma}^p}\d s
	\\&\nonumber\leq C \int_0^t(t-s)^{-\frac{1}{2}}\left[\|\k*\xi_v^\e(s)\|_{\LL_\sigma^\infty}\|\xi_v^\e(s)-\xi_v^0(s)\|_{\L^p}\right.\\&\qquad\nonumber\left.+\|\k*(\xi_v^\e(s)-\xi_v^0(s))\|_{\LL_\sigma^\infty}\|\xi_v^0(s)\|_{\L^p}\right]\d s
	\\&\leq C_p \int_0^t(t-s)^{-\frac{1}{2}}\big(\|\xi_v^\e(s)\|_{\L^p}+\|\xi_v^0(s)\|_{\L^p}\big)\|\xi_v^\e(s)-\xi_v^0(s)\|_{\L^p}\d s,
\end{align}for $p>2$ and  $t\in[0,T]$. For any $P>0$, $\e\in(0,1],\ v\in\PP_2^M, \ \|\xi_0\|_{\L^p}\leq N$ and $T>0$, we define 
\begin{align}\label{6.15}
	\Upsilon_P^\e:=\bigg\{\omega\in\Omega: \sup_{t\in[0,T]}\|\xi_v^\e(t)\|_{\L^p}\leq P, \ \sup_{t\in[0,T]}\|\xi_v^0(t)\|_{\L^p}\leq P \bigg\}.
\end{align}Using H\"older's inequality in \eqref{6.14} and \eqref{6.15}, we arrive at
\begin{align}\label{6.16}
	\|I_{1,\e}(t)\|_{\L^p}^p\chi_{\Upsilon_P^\e} \leq C_{p,T,P}\chi_{\Upsilon_P^\e}\int_0^t(t-s)^{-\frac{1}{2}}\|\xi_v^\e(s)-\xi_v^0(s)\|_{\L^p}^p\d s,
\end{align}where $\chi_A$ is the characteristic function of the set $A$. Let us consider the second term $I_{2,\e}$ in the right hand side of \eqref{6.13} and estimate it using \cite[Lemma 3.1(i)]{IGCR} to obtain 
\begin{align*}
	\|I_{2,\e}(t)\|_{\L^p}\leq C\sum_{j=1}^{n}\int_0^t \|\xi_v^\e(s)-\xi_v^0(s)\|_{\L^p}|v_j(s)|\d s.
\end{align*}H\"older's inequality  and \eqref{6.15} results to 
\begin{align}\label{6.17}\nonumber
	\|I_{2,\e}(t)\|_{\L^p}^p\chi_{\Upsilon_P^\e} &\leq C\sum_{j=1}^{n}\bigg(\int_0^t|v_j(s)|^2\d s\bigg)^\frac{p}{2}\bigg(\int_0^t\|\xi_v^\e(s)-\xi_v^0(s)\|_{\L^p}^2\d s\bigg)^\frac{p}{2}\\&\leq CM^\frac{p}{2}\chi_{\Upsilon_P^\e} \int_0^t\|\xi_v^\e(s)-\xi_v^0(s)\|_{\L^p}^p\d s.
\end{align}Combining \eqref{6.14}-\eqref{6.17} and applying Gronwall's inequality, we deduce 
\begin{align}\label{6.18}\nonumber
	&\sup_{t\in[0,T]}\|	\xi_v^\e(t,\x)-\xi_v^0(t,\x)\|_{\L^p}^p\chi_{\Upsilon_p^\e}\\&\nonumber\leq C\e^\frac{p}{2}\sup_{t\in[0,T]}\|I_{3,\e}(t)\|_{\L^p}^p\chi_{\Upsilon_p^\e} \exp\bigg\{C\int_0^T\big[(t-s)^{-\frac{1}{2}}+M^\frac{p}{2}\big]\d s\bigg\}\\&\leq C\e^\frac{p}{2}\sup_{t\in[0,T]}\|I_{3,\e}(t)\|_{\L^p}^p\chi_{\Upsilon_p^\e}.
\end{align}
Let $\varrho>0$. By Proposition \ref{prop6.9}, there exists a $P>0$ large enough such that 
\begin{align}\label{6.19}\nonumber
	&	\sup_{\|\xi_0\|_{\L^p}\leq N}\sup_{v\in\PP_2^M}\sup_{\e\in(0,1]}\big[1-\P(\Upsilon_P^\e)\big] \\&\nonumber\leq 	\sup_{\|\xi_0\|_{\L^p}\leq N}\sup_{v\in\PP_2^M}\sup_{\e\in(0,1]}\bigg\{\P\bigg(\sup_{t\in[0,T]}\|\xi_v^\e(t)\|_{\L^p}>P\bigg)+\P\bigg(\sup_{t\in[0,T]}\|\xi_v^0(t)\|_{\L^p}>P\bigg)\bigg\} \\&\leq \varrho.
\end{align}From \eqref{6.18} and \eqref{6.19}, for any $\e\in(0,1], \ \xi_0\in\L^p({\1})$ and $v\in\PP_2^M$, we find 
\begin{align}\label{6.20}\nonumber
	&	\sup_{\|\xi_0\|_{\L^p}\leq N} \sup_{v\in\PP_2^M}\P\bigg(\sup_{t\in[0,T]}\|\xi_v^\e(t)-\xi_v^\e(t)\|_{\L^p}>\delta\bigg) \\&\nonumber\leq 	\sup_{\|\xi_0\|_{\L^p}\leq N} \sup_{v\in\PP_2^M}\big[1-\P(\Upsilon_P^\e)\big]\\&\nonumber\quad+	\sup_{\|\xi_0\|_{\L^p}\leq N} \sup_{v\in\PP_2^M}\P\bigg(\sup_{t\in[0,T]}\|\xi_v^\e(t)-\xi_v^\e(t)\|_{\L^p}\chi_{\Upsilon_P^\e}>\delta\bigg)\\&\leq \varrho+	\sup_{\|\xi_0\|_{\L^p}\leq N} \sup_{v\in\PP_2^M}\P\bigg(\sqrt{\e}\sup_{t\in[0,T]}\|I_{3,\e}(t)\|_{\L^p}\chi_{\Upsilon_P^\e}>\frac{\delta}{C}\bigg).
\end{align}Using the linear growth of $\sigma_j$ (see Hypothesis \ref{hyp1}), we have for any $t\in[0,T]$
\begin{align}\label{6.21}\nonumber&
	\sup_{\e\in(0,1]}\E\bigg[\sup_{t\in[0,T]}\|\sigma_j(t,\cdot,\xi_v^\e(t,\cdot))\|_{\L^p}^p\bigg] \\&\leq C\bigg(1+\sup_{\e\in(0,1]}\E\bigg[\sup_{t\in[0,T]}\|\xi_v^\e(t)\|_{\L^p}^p\bigg]\bigg)<\infty.
\end{align}Using \cite[Corollary 3.6]{IGCR}, we obtain for any $\|\xi_0\|_{\L^p}\leq N$ and $v\in\PP_2^M$ the family
\begin{align*}
	\{I_{3,\e}:\e\in(0,1]\}
\end{align*}is tight in $\C([0,T];\L^p({\1}))$, for $p>2$. By Prokhorov's Theorem (see \cite[Section 5]{PB}), we find that for any sequence $\e_m\to0, \ v_m\in\PP_2^M$ and $\|\xi_{0,m}\|_{\L^p}\leq N$, there exists a subsequence (denoted by $(\e_m,v_m,\xi_{0,m})$) such that $I_{3,\e_m}(t,\x)$ converges in distribution. Now, we multiply by the subsequence $\sqrt{\e_m}\downarrow0$, then it is obvious that 
\begin{align*}
	\sqrt{\e_m}\sup_{t\in[0,T]}\|I_{3,\e_m}(t)\|_{\L^p}\chi_{\Upsilon_P^{\e_m}} \Rightarrow 0 \ \ \text{ as } \ \ m\to\infty,
\end{align*}where ``$\Rightarrow$" stands for convergence in distribution. Thus, the original sequence $$\sqrt{\e}\sup\limits_{t\in[0,T]}\|I_{3,\e}(t)\|_{\L^p}\chi_{\Upsilon_P^\e} \Rightarrow 0\ \text{ as }\ \e\to 0.$$ As we can see that the sequential limit is zero, which is a constant, therefore the sequence converges in probability also, that is,
\begin{align*}
	\lim_{\e\to0} \P\bigg(	\sqrt{\e}\sup_{t\in[0,T]}\|I_{3,\e}(t)\|_{\L^p}\chi_{\Upsilon_P^\e}>\frac{\delta}{C}\bigg)=0.
\end{align*}Due to arbitrary choice of $\e_m\downarrow0, \ v_m\in\PP_2^M$, and $\|\xi_{0,m}\|_{\L^p}\leq N$, we deduce 
\begin{align}\label{6.22}
	\lim_{\e\to0}\sup_{\|\xi_0\|_{\L^p}\leq N} \sup_{v\in\PP_2^M}\P\bigg(\|\xi_v^\e-\xi_v^0\|_{\C([0,T];\L^p({\1}))}>\delta\bigg)\leq \varrho.
\end{align}The choice of $\varrho$ is arbitrary, hence \eqref{6.12} follows.
\end{proof}
\begin{proof}[Proof of Theorem \ref{thrm6.7}]
Combining Theorem \ref{thrm6.4} and Corollary \ref{cor6.11}, we obtain the required result.
\end{proof}
Let us now move to the proof of Theorem \ref{thrm6.8}, that is, ULDP in $\C([0,T]\times{\1})$ topology. For that, we need the following intermediate result.
\begin{lemma}\label{lem6.12}
For any $\|\xi_0\|_{\C({\1})}\leq N$, the family 
\begin{align*}
	\{\xi_v^\e-\xi_v^0:v\in\PP_2^M, \e\in(0,1]\}
\end{align*}is tight in $\C([0,T]\times{\1})$.
\end{lemma}
\begin{proof}
From \eqref{6.13}, we have 
\begin{align*}
	\xi_v^\e(t,\x)-&\xi_v^0(t,\x) =: I_{1,\e}+I_{2,\e}+\sqrt{\e}I_{3,\e}.
\end{align*}Using similar arguments as in \eqref{6.14}, we obtain 
\begin{align}\label{6.23}
	\|I_{1,\e}(t)\|_{\L^p}&\leq C_p \int_0^t(t-s)^{-\frac{1}{2}}\big(\|\xi_v^\e(s)\|_{\L^p}+\|\xi_v^0(s)\|_{\L^p}\big)^2\d s. 
\end{align} From Corollary \ref{cor6.10}, we conclude that the terms appearing in the right hand side of the estimate \eqref{6.23} are uniformly bounded in probability over the bounded subsets of initial conditions and for all $\e\in(0,1]$.  Therefore, by \cite[Corollary 3.2]{IGCR}, for all $v\in\PP_2^M$ and $\e\in(0,1]$, the family
\begin{align*}
	\{I_{1,\e}:\|\xi_0\|_{\C({\1})}\leq N\}
\end{align*}are tight in $\C([0,T]\times{\1})$. Using \cite[Corollary 3.6]{IGCR}, we have
\begin{align*}
	\{I_{3,\e}:\|\xi_0\|_{\C({\1})}\leq N, \ v\in\PP_2^M,\ \e\in(0,1]\}
\end{align*}is tight in $\C([0,T]\times{\1})$.

We consider the penultimate term $I_{2,\e}$ and estimate it using \cite[Lemma 3.1(iv)]{IGCR} (cf. \cite[Lemma 2]{MSLS}), with $\gamma=2, \ q>\frac{4}{3}, \ p=\frac{q}{q-1}, \ \kappa_p=1-\frac{1}{q}$ and any $\beta\in\big(0,\beta_0-\frac{1}{2}\big)$, where $\beta_0=\min\big\{2-\frac{2}{q},1\big\},$ for $\x,\x+\y\in{\1}$, $t\in[0,T]$,
\begin{align*}
	|I_{2,\e}(t,\x)-I_{2,\e}(t,\x+\y)|&\leq C|\y|^\beta\bigg(\int_0^t\|\xi_v^\e(s)-\xi_v^0(s)\|_{\L^q}^2|v(s)|^2\d s\bigg)^\frac{1}{2} \\&\leq CM|\y|^\beta \sup_{s\in[0,T]}\|\xi_v^\e(s)-\xi_v^0(s)\|_{\L^q}.
\end{align*}Similarly, the H\"older's continuity in time $t$ follows form \cite[Lemma 3.1(ii)]{IGCR}. For any $q>2$ and $\alpha\in(0,\frac{1}{2}-\frac{1}{q})$, we have
\begin{align*}
	\|I_{2,\e}(t)-I_{2,\e}(s)\|_{\L^\infty} &\leq C|t-s|^\alpha \bigg(\int_0^T\|\xi_v^\e(s)-\xi_v^0(s)\|_{\L^q}^2|v(s)|^2\d s \bigg)^\frac{1}{2}\\&\leq CM|t-s|^\alpha \sup_{ts\in[0<T]}\|\xi_v^\e(s)-\xi_v^0(s)\|_{\L^q}.
\end{align*}For large enough $q$, we observe that $I_{3,\e}$ is H\"older's continuous in both space and time variables and the H\"older norm is uniformly bounded in probability over $\e\in(0,1], \ v\in\PP_2^M$ and $\|\xi_0\|_{\C({\1})}\leq N$. Therefore, the family $\{I_{2,\e}: \ \|\xi_0\|_{\C({\1})}\leq N, \ \e\in(0,1], \ v\in\PP_2^M\}$ is also tight in the space $\C([0,T]\times{\1})$. 
\end{proof}
\begin{theorem}[Convergence in probability in supremum norm]\label{thrm6.13}
For any $T>0,\ \delta>0, \ N>0$ and $M>0$
\begin{align}\label{6.24}
	\lim_{\e\to0}\sup_{\|\xi_0\|_{\C({\1})}}\sup_{v\in\PP_2^M}\P\bigg(\|\xi_v^\e-\xi_v^0\|_{\C([0,T]\times{\1})}>\delta\bigg)=0.
\end{align}
\end{theorem}
\begin{proof}
In Lemma \ref{lem6.12}, we established that for any $\|\xi_0\|_{\C({\1})}\leq N$, the family 
\begin{align*}
	\{\xi_v^\e-\xi_v^0:v\in\PP_2^M, \e\in(0,1]\}
\end{align*}is tight in $\C([0,T]\times{\1})$. Again, by Prokhorov's Theorem (see \cite[Section 5]{PB}), we obtain that for any sequence $\e_m\to0,\ v_m\in\PP_2^M$ and $\|\xi_{0,m}\|_{\C({\1})}\leq N$, there exists a subsequence (denoted by $(\e_m,v_m,\xi_{0,m})$) such that $\xi_v^\e-\xi_v^0$ converges in distribution. Let $\xi^*$ be the limit. 

Since we have the embedding $\C([0,T]\times{\1})\hookrightarrow\C([0,T];\L^p({\1}))$, the above convergence gives $\xi_v^\e-\xi_v^0\Rightarrow \xi^*$, in the space $\C([0,T];\L^p({\1}))$ topology also. Therefore, Corollary \ref{cor6.11} yields
\begin{align*}
	\sup_{t\in[0,T]}\|\xi^*(t)\|_{\L^p}^p=0,
\end{align*}with probability 1 which implies that $\|\xi^*\|_{\C([0,T]\times{\1})}=0$ with probability 1. Since $\|\xi_v^\e-\xi_v^0\|_{\C([0,T]\times{\1})}$ is converging to $0$, we have 
\begin{align*}
	\lim_{m\to0}\P\big(\|\xi_{v_m}^{\e_m}-\xi_{v}^0\|_{\C([0,T];{\1})}>\delta\big)=0.
\end{align*}Due to arbitrary choice of $\e_m\downarrow0, \ v_m\in\PP_2^M$, and $\|\xi_{0,m}\|_{\L^p}\leq N$, we obtain the required result \eqref{6.24}.
\end{proof}
\begin{proof}[Proof of Theorem \ref{thrm6.8}]
Combining Theorems \ref{thrm6.4} and \ref{thrm6.13}, we obtain the required result.
\end{proof}

\medskip\noindent
\textbf{Acknowledgments:} The first author would like to thank Ministry of Education, Government of India - MHRD for financial assistance.  M. T. Mohan would  like to thank the Department of Science and Technology (DST) Science $\&$ Engineering Research Board (SERB), India for a MATRICS grant (MTR/2021/000066).

\end{document}